\documentclass[a4paper,11pt]{amsart}

\usepackage{amssymb}
\usepackage[initials]{amsrefs}

\usepackage{youngtab}

\usepackage[hmargin=3.5cm,foot=0.7cm]{geometry}
\usepackage[utf8x]{inputenc}

\usepackage[english]{babel}

\usepackage{tikz}
\usetikzlibrary{matrix,arrows}

\BibSpec{article}{%
  +{}{\PrintAuthors}  		{author}
  +{,}{ \textit}     		{title}
  +{,}{ }             		{journal}
  +{}{ \textbf}       		{volume}
  +{}{ \parenthesize} 		{date}
  +{,}{ }      	      		{conference}
  +{,}{ }      	      		{book}
  +{,}{ }            		{pages}
  +{,}{ }            	 	{note}
  +{,}{ }            	 	{status}
  +{,}{  \texttt } {eprint}
  +{.}{}              {transition}
}
\BibSpec{book}{%
  +{}{\PrintAuthors}  {author}
  +{,}{ \textit}      {title}
  +{,}{ } 	      {series}
  +{,}{ }	      {note}
  +{,}{ }             {publisher}
  +{,}{ }             {place}
  +{,}{ }             {date}
  +{.}{}              {transition}
}

\newtheorem{theorem}{Theorem}[section]
\newtheorem*{theorem*}{Theorem}

\theoremstyle{plain}

\newtheorem{lemma}[theorem]{Lemma}
\newtheorem{proposition}[theorem]{Proposition}
\newtheorem{conjecture}[theorem]{Conjecture}

\newtheorem*{angularity}{Angularity conjecture}
\newtheorem*{lemma*}{Lemma}
\newtheorem*{question*}{Question}

\theoremstyle{definition}

\newcommand{\CC}{\mathbb{C}}

\newcommand{\PP}{\mathbb{P}}
\newcommand{\RR}{\mathbb{R}}

\newcommand{\calA}{\mathcal{A}}

\newcommand{\calC}{\mathcal{C}}

\newcommand{\calF}{\mathcal{F}}

\newcommand{\calK}{\mathcal{K}}
\newcommand{\calL}{\mathcal{L}}

\newcommand{\calV}{\mathcal{V}}

\newcommand{\dt}{\left.\frac{d}{dt}\right|_{t=0}}

\DeclareMathOperator{\vol}{vol}
\DeclareMathOperator{\id}{id}

\DeclareMathOperator{\glob}{glob}
\DeclareMathOperator{\Grass}{Gr}
\DeclareMathOperator{\Sym}{Sym}

\DeclareMathOperator{\im}{Im}

\DeclareMathOperator{\Curv}{Curv}

\newcommand{\largewedge}{\mbox{\Large $\wedge$}}

\DeclareMathOperator{\codim}{codim}

\makeatletter
\DeclareRobustCommand{\pder}[1]{%
  \@ifnextchar\bgroup{\@pder{#1}}{\@pder{}{#1}}}
\newcommand{\@pder}[2]{\frac{\partial#1}{\partial#2}}
\makeatother

\title[]{Classification of angular curvature measures and a proof of the angularity conjecture}

\author{Thomas Wannerer}
\email{thomas.wannerer@uni-jena.de}
\address{Fakult\"at f\"ur Mathematik und Informatik, Friedrich-Schiller-Universit\"at Jena, 07743 Jena, Germany}
\date{\today}
\subjclass[2010]{53C65, 52A38}
\thanks{Supported by DFG grant WA 3510/1-1.}
\begin{document}

\begin{abstract}
This article is concerned with the interplay between the theory of smooth valuation on manifolds and Riemannian geometry. 
We confirm the angularity conjecture formulated by A.~Bernig, J.H.G.~Fu, and G.~Solanes which sheds new light on the geometric meaning of the 
Lipschitz-Killing valuations. The proof relies on a complete classification of translation-invariant 
angular curvature measures on $\RR^n$, a result  of independent interest. 
\end{abstract}

\maketitle

\section{Introduction}

Over the last decade it has become clear that the theory of valuations on convex bodies, a classical line of research in convex geometry, admits a natural 
continuation  in the setting of  general smooth manifolds. 
According to the pioneering work of Alesker \cites{valmfds_survey,valmfdsI,valmfdsII,valmfdsIII,valmfdsIV,valmfdsIG}, to each  smooth manifold $M$ is associated the commutative filtered algebra of smooth valuations $\calV(M)$ on $M$,
with the Euler characteristic $\chi$ as multiplicative identity. Loosely speaking, smooth valuations on $M$  are finitely additive set functions satisfying a smoothness condition 
and the Alesker product of valuations reflects the operation of intersection of subsets of $M$.  It was soon realized 
 that this new structure can be used to solve classical problems in integral geometry: the description of explicit kinematic formulas in complex space forms---a problem first taken up by Blaschke and his school \cites{blaschke39,rohde40,varga39} in the 1930s with special cases solved 
by Santal\'o \cites{santalo52}, Gray \cite{gray_tubes}, Shifrin \cite{shifrin81}, and others---had to wait until 2015
when it was finally found by Bernig, Fu, and Solanes \cites{hig,bfs} using the new tools from  valuation theory introduced by Alesker.

This paper is concerned with the interplay between valuation theory and  Riemannian geometry. 
This new line of research originated from the observation that in order to identify structures that help in the investigation of the integral geometry of all 
isotropic spaces  it is fruitful to study canonical classes of valuations on general Riemannian manifolds (see \cite{fw_Riemannian}).
The intrinsic volumes of a convex body play a fundamental role in convex geometry. Their extension to   Riemannian
manifolds, the Lipschitz-Killing valuations, will be our primary focus in this paper.  Their existence is non-trivial: 
Alesker first observed  that it follows from  a classical theorem of H.~Weyl \cite{weyl_tubes} on the volume of tubes in combination with the Nash embedding theorem; an alternative 
approach in the spirit of Chern's intrinsic proof of the Gauss-Bonnet theorem can be found in \cite{fw_Riemannian}.  
The closely related Lipschitz-Killing curvatures of $M$ have  remarkable properties; they arise for example 
in the asymptotic expansion of the trace of the heat kernel \cite{donnelly} and they converge  under approximations of a Riemannian manifold by 
a piecewise linear one \cite{cms}. The conjectures 
presented in the next paragraph shed new light on the geometric meaning of the Lipschitz-Killing valuations.

Smooth valuations may be localized, albeit non-uniquely. The resulting space of smooth curvature measures on $M$, 
denoted  by $\calC(M)$, is naturally a module over $\calV(M)$ with respect to the Alesker product. 
Bernig, Fu, and Solanes \cite{bfs} observed that a Riemannian metric on  $M$ induces a canonical isomorphism 
$$\tau\colon \calC(M)\to \Gamma(\Curv(TM))$$
between smooth curvature measures 
on $M$ and smooth sections of the bundle of translations-invariant smooth curvature measures on the  tangent spaces of $M$. 
Following \cite{bfs} this allows us to make the following definition: a curvature measure on $M$ is  called angular 
if  $\tau_p\Phi$ is angular for every point $p\in M$.
Here a translation-invariant curvature measure $\Psi$ on $\RR^n$ is called angular if there 
exist functions $f_k$ on the Grassmannian $\Grass_k(\RR^n)$ of $k$-dimensional linear subspaces in $\RR^n$
such that  
\begin{equation}\label{eq:angular_nonhom}\Psi (P,U)= \sum_k \sum_{\dim F=k} f_k(\overline F) \gamma(F,P)\vol_k(F\cap U)\end{equation}
for every polytope $P\subset \RR^n$ and Borel set $U\subset \RR^n$, where 
the first sum extends over all integers $k=0,1,\ldots,n$,  the  second sum is over all $k$-faces of $P$, 
$\overline F$ is the translate of the  affine hull of $F$ containing the origin, and
$\gamma(F,P)$ is the external angle of $P$ at the face $F$.  
Let $\calA(M)$ denote the space of angular curvature measures on $M$. It may seem surprising, but there are natural curvature measures, e.g., arising in hermitian integral geometry \cite{hig}, that are not angular.

Motivated by their results on the integral geometry of complex space forms,   Bernig, Fu, and Solanes  \cite{bfs}  formulated the following

\begin{angularity}
 Let $M$ be a Riemannian manifold. Then $\calA(M)$ is invariant under the action of the Lipschitz-Killing algebra,
$$\calL\calK(M)\cdot \calA(M)\subset \calA(M)$$
\end{angularity}

Here $\cdot$ denotes the Alesker product. We call a valuation $\mu$ 
angular if $\mu\cdot \calA(M)\subset \calA(M)$. The angularity conjecture states that the Lipschitz-Killing valuations are angular. Bernig, Fu, and Solanes \cite{bfs}
conjecture that this property even characterizes the Lipschitz-Killing valuations:

\begin{conjecture}\label{conj:angularity2}
 The algebra of angular valuations on $M$ equals $\calL\calK(M)$.  
\end{conjecture}

In the presence  of additional invariance assumptions the angularity conjecture  is known to be true in the following special cases: 
translation-invariant curvature measures on $\RR^n$  and isometry-invariant 
curvature measures in complex projective space $\CC P^n$. Both results are contained in  \cite{bfs}. 
Also for Conjecture~\ref{conj:angularity2} the integral geometry of complex space forms provides evidence \cite{bfs_dual}.

The main result of this paper is 
\begin{theorem}\label{conj:angularity}
 The angularity conjecture is true. 
\end{theorem}

The basic idea of the proof of Theorem~\ref{conj:angularity} is to  reduce the general  case to  $M=\RR^n$ by first showing that the class of angular curvature 
measures is invariant under pullback by isometric immersions. 
The latter is an immediate consequence of the following result that seems to be  also 
 of independent interest.  We call a smooth curvature measure $\Psi$ on $M$ angular at $p$ if $\tau_p\Psi$ is angular.

\begin{theorem} \label{thm:pullback}
 Let $f\colon M\to \overline M$ be an isometric immersion of Riemannian manifolds and let $p\in M$.  
 If $\Psi \in \calC(\overline M)$ is angular at $f(p)$, then  $f^*\Psi$ is angular at $p$.  
\end{theorem}

The proofs of Theorems~\ref{conj:angularity} and \ref{thm:pullback} rely on a complete classification of translation-invariant angular curvature measures on $\RR^n$,
Theorem~\ref{thm:ccc} below.
The analogous question for translation-invariant valuations on $\RR^n$ was asked in  \cite{crm_bcn}*{Problem 2.3.13}. In the case of curvature measures the 
problem  is equivalent to characterizing those functions $f$ on $\Grass_k(\RR^n)$ for which the weighted sums 
\begin{equation}\label{eq:angularity}\Phi(P,U)= \sum_{\dim F=k} f(\overline F) \gamma(F,P)\vol_k(F\cap U),\end{equation}
where  $P\subset\RR^n$ is a polytope and   $U\subset \RR^n$ a Borel set, extends to a  translation-invariant smooth curvature measure on $\RR^n$.

It is   a remarkable fact that many constructions of  central importance to  convex geometry associating to a convex body  $K\subset \RR^n$ an object such as a number, a measure or another convex body, 
all have the property of being valuations, i.e.,  additive in the sense that 
$$\Phi(K\cup L)+\Phi(K\cap L)= \Phi(K)+ \Phi(L)$$
whenever $K\cup L$ is again convex; see, e.g., \cites{ludwig_minkowski,ludwig_intersection,ludwig_HT, lr_affine, lyz_proj,lz_intersect, sw_even,sw_gen,schneider_kinematische,schneider_curvature, klain_short}.
Arguing as in  \cite{mcmullen_valuations} it is not difficult to see that  the expression \eqref{eq:angularity} is a valuation  for any function $f$.

Clearly, every angular curvature measure \eqref{eq:angularity} is even in the sense that 
$\Phi(-P,-U)= \Phi(P,U)$.
It is not difficult to see that for $k=n-1$ this is the only restriction  and thus \eqref{eq:angularity} extends for every  smooth function $f$  to a  smooth curvature measure.
Choosing $f$ to be constant yields the well-known curvature measures introduced by Federer \cite{federer}. Further examples of angular curvature measures of degree $k<n-1$ are harder to come by. As observed by Bernig, Fu, and Solanes \cite{bfs}*{Lemma 2.30}, 
a whole family of examples can be constructed as follows: if $\omega\in \largewedge^n(\RR^n\oplus\RR^n)^*\subset \Omega^n(\RR^n\oplus \RR^n)$ is a 
constant coefficient form, then 
$$\Phi(P,U)=  \int_{N_1(P)\cap \pi^{-1}(U)} \omega$$
is an angular smooth curvature measure. Here $N_1(P)\subset \RR^n\times D^n$ is the normal disc current of $P$ formed by the outward normals of $P$ of length at most $1$
(see Section~\ref{subsec:ccc} for details).
Following Bernig, Fu, and Solanes \cite{bfs} we call such curvature measures constant coefficient curvature measures. 
In the sense of currents $N_1(A)$ exists for a wide class of subsets of $A\subset \RR^n$, see \cites{fu_subanalytic,pr_wdc,fpr_wdc}.

Let  $\widetilde\Grass_k(\RR^n)$ denote the oriented Grassmannian, the manifold of oriented $k$-dimensional linear subspaces of $\RR^n$. 
We call a function on $\widetilde\Grass_k(\RR^n)$ even if it is invariant under change of orientation. 
Note that even functions on the oriented Grassmannian $\widetilde\Grass_k(\RR^n)$ correspond bijectively to functions on  $\Grass_k(\RR^n)$. 
The oriented Grassmannian smoothly embeds into the exterior power $\largewedge^k\RR^n$ as $E\mapsto \vec E$, where $\vec E=e_1\wedge \cdots \wedge e_k$ 
for some positively oriented orthonormal basis of $E$. This map is called the Pl\"ucker embedding.

\begin{theorem} \label{thm:ccc} Let $0\leq k <n-1$  be an integer and $f$ be a function on $\Grass_k(\RR^n)$.  Then 
\eqref{eq:angularity}
extends to a  translation-invariant smooth curvature measure on $\RR^n$ if and only if $f$
is the restriction of a $2$-homogeneous polynomial to the image of the Pl\"ucker embedding. \
Consequently, the space of translation-invariant angular curvature measures of degree $k$ has dimension 
 $$\frac{1}{n-k+1}\binom{n}{k}\binom{n+1}{k+1}$$
 and coincides with the space of constant coefficient curvature measures. 
\end{theorem}

\section{Preliminaries}

\subsection{Convex geometry}\label{sec:cones}
For later use we collect here  several facts  on convex cones and external angles. An excellent reference for this material is \cite{schneiderBM}.

\subsubsection{Intrinsic volumes} The intrinsic volumes of a convex body $A\subset \RR^n$  arise (up to normalization) in Steiner's formula 
\begin{equation}\label{eq:steiner} \vol_n(A_\varepsilon)= \sum_{k=0}^n \omega_{n-k} \varepsilon^{n-k} V_k(A), \qquad \varepsilon>0,\end{equation}
for the volume of an $\varepsilon$-neighborhood of a convex body. Here $\omega_i$ denotes the volume of the $i$-dimensional Euclidean unit ball. If $A\subset \RR^n$ is more generally a compact set of positive reach, then, as was shown by Federer~\cite{federer},  the expansion \eqref{eq:steiner} persits  for sufficiently small $\varepsilon>0$ and defines the intrinsic volumes of such $A$. 

\subsubsection{Convex cones}

Let $V$ be a finite-dimensional real vector space. If  $C\subset V$ is a closed convex cone, 
then  the largest linear subspace contained in $C$ is called the lineality space of $C$ and denoted by  $L(C)$. 
It is not difficult to see that $L(C)=C\cap (- C)$. If $L(C)=\{0\}$ then the cone is called pointed.  The polar cone to $C$ is
$$C^\circ =\{\xi\in V^* \colon \xi(x)  \leq 0  \ \text{for}\ x\in C\}.$$ 
A face of a convex set $A\subset V$ is a convex subset $F\subset A$ such
that  $x,y\in A$ and $(x + y)/2 \in F$ implies $x, y\in F$.
An extreme ray of $C$ is a ray that is also a face of $C$. An $n$-dimensional convex cone is called simplicial if it has precisely $n$ extreme rays. A cone is called polyhedral  if it is spanned by finitely many rays.

\begin{lemma}\label{lemma:triangulation}
For every  polyhedral cone  $C\subset V$ there exist  closed convex cones
$C_1,\ldots, C_m\subset V$ that have the same lineality space $L$ as $C$ such that 
$$C=C_1\cup \cdots \cup C_m,$$
each $pr_{V/L}(C_i)$ is simplicial where $pr_{V/L}\colon V\to V/L$ denotes the canonical projection, and
each   intersection $C_i\cap C_j$ is a face of both $C_i$ and $C_j$.
\end{lemma}
\begin{proof}
Note that the cone $pr_{V/L}(C)$ is pointed. Hence there exists a hyperplane $H$ in $V/L$ such that $P=H\cap pr_{V/L}(C)$ is a convex polytope that spans $pr_{V/L}(C)$. Write $P=S_1\cup \cdots \cup S_m$ as a union of simplices such that $S_i\cap S_j$ 
is a face both $S_i$ and $S_j$.  Let $C'_i\subset V/L$ be the cone spanned by $S_i$. The cones $C_i=pr_{V/L}^{-1}(C_i')$ have the desired property. 
\end{proof}

The tangent cone of a polytope $P\subset V$ at $x\in P$ is the closed convex cone
$$T_x P = \mathrm{cl}\bigcup_{t>0} t(P-x)$$
The lineality space of   $T_x P$ is the linear subspace generated by the unique face $F$ of $P$ 
containing $x$ in its relative interior. Note also that $T_xP = T_y P $ if $x,y$ belong to the relative interior of the same face $F$;
we denote this common cone by $T_FP$.

\subsubsection{The external angle}
Let $\langle x ,y\rangle$ denote the Euclidean inner product on $\RR^n$, $|x|$ the corresponding norm, and $S^{n-1}$ the unit sphere. 
The external angle  of a closed convex cone $C\subset \RR^n$ equals the fraction of $L(C)^\perp$ taken up by $C^\circ$. More precisely, 
\begin{equation}\label{eq:exterior_angle} \gamma(C)=\frac{\vol_{n-k-1}(S^{n-1}\cap C^\circ)}{\vol_{n-k-1}(S^{n-k-1})}, \qquad \text{with }k=\dim L(C)\end{equation}
The external angle of a polytope $P$ at a face $F$ is denoted $\gamma(F,P)=\gamma(T_FP)$.
An important property of the external angle is that it is independent of the ambient space: 
If $C\subset \RR^m\subset \RR^n$, then $\gamma(C)$ is the same whether computed in $\RR^m$ or $\RR^n$. 

It is a well-known fundamental fact, see, e.g.,  \cite{sw_sig}*{Theorem 6.5.5}, that 
there exist constants $c_0,\ldots,c_n$  such that 
$$\vol_{n-1}(S^{n-1}\cap C^\circ)= \sum_{k=0}^{n-1} c_k V_k(C\cap S^{n-1})+c_n$$
 for every closed convex cone $C\subset \RR^n$. This was first proved by McMullen~\cite{mcmullen_angle-sum} and was apparently independently discovered by 
 Milnor \cite{milnor}*{p. 213}.
It follows that the external angle is continuous if the distance between two cones $C_1,C_2$ is defined as 
the Hausdorff distance between $C_1\cap S^{n-1}$ and $C_2\cap S^{n-1}$ and that the external angle is finitely additive in the sense that 
\begin{equation}\label{eq:inclusion_exclusion}\gamma(C_1\cup\cdots \cup C_m)= \sum_{j=1}^m(-1)^{j-1}  \sum_{1\leq i_1<\cdots <i_j\leq m}\gamma(C_{i_1}\cap \cdots \cap C_{i_j}). \end{equation}
 whenever $C_1,\ldots, C_m$ have the same lineality space and $C_1\cup\cdots \cup C_m$ is convex.

\subsection{Curvature measures and valuations} \label{sec:val}

Although the theory of smooth valuations on a manifold $M$ developed by Alesker \cites{valmfdsI,valmfdsII,valmfdsIII,valmfdsIV,valmfdsIG}
is entirely independent of  orientation or orientability,
it will be convenient to assume that $M$ is oriented.
Since Theorem~\ref{conj:angularity} and Theorem~\ref{thm:pullback} are  local statements, this will result in no loss of generality.

\subsubsection{Normal cycles and smooth valuations} Let $M$ be an oriented smooth manifold of dimension $n$. 
 The (co-)normal cycle of a closed submanifold with corners  $P\subset M$  is  a 
canonically oriented Lipschitz submanifold of the cosphere bundle $SM$ 
of $M$; as a set 
 $$ N(P) = \coprod_{p\in P} (T_p P)^\circ\setminus\{0\}/\sim, $$
 where $T_p P\subset T_pM$ denotes the tangent cone of $P$ at $p$, $(T_p P)^\circ\subset T_p^*M$ is its polar cone, and
 $\xi\sim \eta$ if and only if $\xi = \lambda \eta$ for some positive real number $\lambda$.
 If $M$ carries a Riemannian metric we will usually work with the sphere bundle instead of cosphere bundle. 
  An important property of the normal cycle is finite additivity  when regarded as an $(n-1)$-current,
 \begin{equation}\label{eq:nc add}N(A\cup B)= N(A) + N(B) - N(A\cap B)\end{equation}
where the existence of normal cycles for any three of $A, B, A \cap B, A \cup B$ implies its existence
for the fourth.

To each pair $(\psi, \phi)\in \Omega^{n-1}(SM) \oplus \Omega^{n}(M)$  one may assign the  smooth curvature measure $\Psi$ that associates to every closed
submanifold with corners $P$ the signed measure 
\begin{equation}\label{eq:def curv}
\Psi(P,U)= \int_{N(P)\cap \pi^{-1}(U)} \psi  +  \int_{P\cap U} \phi,
\end{equation}
where $U\subset M$ is a Borel subset and $\pi\colon SM\to M$ denotes the canonical projection. 
The space of all such curvature measures on $M$ is denoted by $\calC(M)$.
Such a pair $(\psi,\phi)$ determines also  a  valuation given by $\mu(P)=  \int_{N(P)} \psi  +  \int_{ P} \phi $ for compact submanifolds with corners $P$.
The space of all such set functions is denoted by $\calV(M)$. We denote $\glob\colon \calC(M)\to\calV(M)$ the globalization map  $(\glob\Psi)(P)=\Psi(P,P)$.

It was first observed by Z\"ahle   \cite{zahle} that the  Federer's curvature measures may  be described in terms of 
canonical  invariant differential forms $\kappa_0,\kappa_1,\ldots,\kappa_{n-1}$ on the sphere bundle of $\RR^n$,
$$\Phi_i(P,U) = \int_{N(P)\cap \pi^{-1}(U)} \kappa_i.$$ 
Thus they are smooth curvature measures and their globalizations $V_i= \glob \Phi_i$, the intrinsic volumes, are smooth valuations.

The normal cycle  is a device from Geometric Measure Theory to extend Federer's theory of curvature measures \cite{federer} to sets
too singular for tubular approximation to work \cites{fu_subanalytic, fu_sandbjerg,pr_wdc,fpr_wdc}. As a consequence, smooth valuations and curvature measures 
may be evaluated on subsets of $M$ much more general than compact submanifolds with corners. 
Smooth valuations 
and curvature measures on $\RR^n$ may  in particular  be evaluated on convex bodies and  turn out to be continuous  valuations in the classical sense of convex geometry.
We have the following deep result of Alesker: 
\begin{theorem}[\cite{alesker_irred}]
 Any translation-invariant continuous valuation on convex bodies in $\RR^n$ may be approximated uniformly on compact sets by translation-invariant smooth valuations. 
\end{theorem}

\subsubsection{Algebraic structures on smooth valuations}

One of the salient features of the space of smooth valuations on a manifold is that it has naturally the structure of a filtered, commutative algebra with a multiplicative identity. 

If $f:M \to \overline M$ is a smooth  embedding then the pullback maps 
$\calC(\overline M) \to \calC(M)$ and $\calV(\overline M) \to \calV(M)$, both of which we denote by $f^*$, are defined by
\begin{equation}\label{eq:pullback}
(f^*\Psi) (P, U) := \Psi(f(P),f(U)),\quad (f^*\mu)(P): = \mu(f(P)).
\end{equation}
Since smooth valuations and curvature measures are sheaves and locally every smooth immersion is an embedding, this defines the pullback also for smooth immersions.
The main properties of the Alesker product of valuations that we use in this paper are summarized in the following Theorem. For a gentler introduction to the subject we 
recommend \cite{crm_bcn}.

\begin{theorem}[\cites{valmfdsIII, valmfdsIG,ab_product, bfs, fu_intersection}]\label{thm:product} Let $M$ be a smooth manifold. 
\begin{enumerate}
\item The space $\calV(M)$ admits a natural commutative multiplication (the Alesker product), 
with the Euler characteristic $\chi$ acting as the multiplicative identity. 
Furthermore $\calV(M)$ acts on $\calC(M)$ in a natural way, compatible with the product of valuations, i.e.,
if $\mu \in \calV(M), \Psi \in \calC(M)$ then $\glob(\mu \cdot \Psi) = \mu \cdot\glob(\Psi)$. If $f$ is a smooth  immersion
as above then $f^*$ is an algebra and module homomorphism, i.e. if $\mu, \nu \in \calV(\overline M), \Psi \in \calC(\overline M)$ then
$$
(f^*\mu)\cdot (f^*\nu) = f^*(\mu \cdot \nu), \quad (f^*\mu)\cdot (f^*\Psi) = f^*(\mu \cdot \Psi). 
$$
\item \label{item:product kinematic} Suppose $X\subset M$ is a compact submanifold with corners, and $T\times M \to M$ 
is a smooth proper family of diffeomorphisms $\varphi_t:M\to M, \ t \in T$, equipped with a smooth measure $dt$. 
Suppose further that the map $T \times S^*M \to S^*M$, induced by the derivative maps $\varphi_{t*}:S^*M \to S^*M$, is a submersion.
 Then $\mu(P)= \int_T \chi( \varphi_t(X) \cap P) \, dt$ defines a smooth valuation on $M$. Given $\nu \in \calV(M), \Psi \in \calC(M)$ we have
\begin{align*}
(\mu \cdot \nu) (P) &= \int_T  \nu(\varphi_t(X) \cap P) \, dt, \\
(\mu \cdot \Psi) (P,E)  &= \int_T  \Psi(\varphi_t(X) \cap P,E) \, dt. 
\end{align*}

\end{enumerate}
\end{theorem}

The Crofton formula \cite{federer}, a classical fact from integral geometry, expresses the $k$th intrinsic volume of a compact submanifold
with corners $P\subset \RR^n$ as 
an integral over the affine Grassmannian of $(n-k)$-planes in $\RR^n$
$$ V_k(P) = \int_{\overline{\Grass}_{n-k}} \chi(P\cap \overline E) \; d\overline E,$$
where $d\overline E$ is a suitably normalized Haar measure.

The next lemma is an immediate consequence of  Theorem~\ref{thm:product}, item \eqref{item:product kinematic}. 
Since it is important for our proof of Theorem~\ref{conj:angularity}, but
does not seem to have been explicitly stated in the literature, we give a proof. 
\begin{lemma} For every smooth curvature measure  $\Psi\in\calC(\RR^n)$
\begin{equation}\label{eq:muk action} (V_k \cdot \Psi) (P,U)=\int_{\overline{\Grass}_{n-k}} \Psi(P\cap \overline E, U) \; d\overline E\end{equation}
for all compact submanifolds with corners $P$ and Borel subsets $U\subset \RR^n$.
In particular, up to a constant $V_k$ equals  the $k$th power of $V_1$.
\end{lemma}

\begin{proof}Let $D^{n-k}\subset \RR^{n-k}\subset \RR^{n-k}\oplus \RR^k$ denote the unit disc and put 
$$\varphi(P)= \int_{SO(n)\times \RR^k} \chi(P\cap g(D^{n-k} +x))\, dg\,dx.$$
By Theorem~\ref{thm:product}, item \eqref{item:product kinematic}, we have $\varphi\in \calV(\RR^n)$. 
One immediately checks that 
$\varepsilon^{-i} \varphi(\varepsilon P)\to 0$ for every integer $0\leq i<k$ and  
$\varepsilon^{-k} \varphi(\varepsilon P)\to V_k(P)$
as $\varepsilon\to 0$. Repeating the argument of \cite{sw_spheres}*{Section 3}, it follows that $\varepsilon^{-k}\varphi(\varepsilon\,\cdot\,)\to V_k$ in $\calV(\RR^n)$. 
By the continuity of the Alesker product and Theorem~\ref{thm:product}, item \eqref{item:product kinematic} we have
\begin{align*}(V_k\cdot \nu) (P)& = \lim_{\varepsilon\to 0} \int_{SO(n)\times \RR^k} \nu(P\cap g(\varepsilon^{-1} D ^{n-k} +x))\, dg\,dx\\
& = \int_{\overline{\Grass}_{n-k}} \nu(P\cap \overline E) \; d\overline E
\end{align*}
for every smooth valuation $\nu$. The corresponding statement for curvature measures now follows by approximating the indicator function of $U$ by smooth functions, see \cite{bfs}.  
\end{proof}

\subsubsection{Intrinsic volumes on Riemannian manifolds}

A fundamental fact  from convex geometry is that  the intrinsic  volumes of a convex body are independent of the ambient space: If $K\subset \RR^m\subset \RR^n$
 then  $V_i(K)$ is the same whether computed in $\RR^m$ or $\RR^n$, see, e.g., \cite{klain_rota}. A generalization of this  is the theorem Weyl \cite{weyl_tubes}
 that  the  intrinsic volumes of a compact submanifold with corners $P$ depend only  on the  metric induced on $P$. As first observed by Alesker  \cite{valmfds_survey}, this in combination 
 with the embedding theorem of Nash, yields the existence of smooth valuations $V_i^M\in\calV(M)$ on $M$, called intrinsic volumes or Lipschitz-Killing valuations,
 characterized by the property that 
 for every isometric embedding $f\colon M\to \RR^N$
 $$V_i^M  = f^*V_i,$$
 where $V_i$ the $i$th intrinsic volume on $\RR^N$. Let $\calL\calK(M)$ denote the span of the intrinsic volumes on $M$. 
 Since the pullback of valuations is a morphism of algebras, we have the following
 \begin{theorem}[\cite{valmfds_survey}]
  The map $t\mapsto V_1^M$ descends to an isomorphism of algebras from the truncated polynomial algebra $\RR[t]/ (t^{\dim M+1})$  to $\calL\calK(M)$.  
 \end{theorem}

 For a construction of the intrinsic volumes on a Riemannian manifold that does not rely on the existence of isometric embeddings of $M$ into Euclidean space 
 and a generalization of the above result, see 
 \cite{fw_Riemannian}; for applications to integral geometry of isotropic spaces see \cites{bfs,fw_Riemannian, sw_spheres}; 
 for a similar  notion in  pseudo-Riemannian and
 contact geometry see \cites{bf_indefinite, faifman_crofton, faifman_contact}.

\subsection{Fiber integration} 

 Let $\pi\colon  E\to B$ be a smooth fiber bundle with fibers $F_p= \pi^{-1}(p)$. Let $n$ denote the dimension of the base $B$ and $r$ the dimension of the fibers. 
 If $\alpha$ is a smooth differential form on $E$ of degree $k\geq r$,
 then the restriction of  $\alpha$ to the fiber $F_p$ is a smooth $r$-form $r_p(\alpha)$ on $F_p$ with values in $\largewedge^{k-r} T_p^*B$ defined by
 $$r_p(\alpha)(X_1,\ldots, X_{k-r})= \left.\left( \widetilde X_1\wedge \cdots \wedge \widetilde X_{k-r} \lrcorner \alpha \right) \right|_{F_p}$$
 where $X_1,\ldots, X_{k-r}\in T_p B$ and $\widetilde X_i$ is an arbitrary lift of $X_i$ to $F_p$. 
 If $\alpha$ has degree degree less than $r$, then $r_p(\alpha)=0$ by definition. 
 
If $\pi_E\colon E\to B$ and $\pi_F\colon F\to C$ are fiber bundles with fibers of the same dimension, and $\varphi\colon E\to F$ is a bundle map,
i.e., $\pi_F\circ \varphi 
= f\circ \pi_E$ for some smooth map $f\colon B\to C$, then 
\begin{equation}\label{eq:restriction_bundle_map} (\varphi^* \otimes f^* )(r _{f(p)}\alpha) = r_p(\varphi^*\alpha).\end{equation}

 Let $P\subset M$ be a compact submanifold with corners of  dimension $d$.  By definition, there exists 
 for each $q\in P$  an open neighorhood $U$ of $q$ in $M$ and 
 a diffeomorphism $\varphi \colon U\to \RR^n$ such that 
 $\varphi(q)=0$ and 
 $$\varphi(P\cap U)= [0,\infty)^{d-k}\times \RR^{k}\times 0_{\RR^{n-d}}.$$
 The integer $0\leq k\leq n$ is called the type of $q\in P$; the type of a point does not depend on the choice of chart.
 
 \begin{lemma}\label{lemma:fiber_int} Let $M$ be a Riemannian manifold, $P\subset M$ a compact submanifold with corners, and let $F\subset M$ be an oriented $k$-dimensional embedded submanifold consisting only of points of $P$ of type $k$.
Let $j\colon \nu F\to SM$ denote the inclusion of the unit normal bundle to $F$ into $SM$. 
If   $\omega\in \Omega^{n-1}(SM)$, then 
$$ \int_{N(P)\cap\pi^{-1}(p)} r_p(j^*\omega)$$
depends smoothly on $p\in F$ and 
 $$\int_{N(P)\cap \pi^{-1}(F)} \omega = \int_F  \int_{N(P)\cap\pi^{-1}(p)} r_p(j^*\omega).$$
Here $r_p(j^*\omega)$ denotes the 
restriction of $j^*\omega$ to the fibers of $\nu F$. 
\end{lemma}
\begin{proof}
 Let $N^*(P)\subset S^*M$ denote the conormal cycle  of $P$. 
 Since a choice  of a Riemannian metric on $M$ induces a bundle isomorphism between normal and conormal bundles  it suffices by \eqref{eq:restriction_bundle_map} to prove
 the lemma for the  conormal cycle of $P$.

 Suppose $P$ has dimension $d$.  For each $q\in F$  there exist an open neighborhood $U$ of $q$ and 
 a diffeomorphism $\varphi \colon U\to \RR^n$ such that 
 $\varphi(q)=0$ and 
 $$\varphi(P\cap U)= [0,\infty)^{d-k}\times \RR^{k}\times 0_{\RR^{n-d}}.$$ Moreover, 
 since the type of a point of $P$ is invariant under diffeomorphisms, we have 
 $$\varphi(F\cap U)\subset 0_{\RR^{d-k}}\times \RR^{k}\times 0_{\RR^{n-d}}.$$
 
 It suffices to prove the lemma for forms $\omega$ compactly supported in $U$. We denote by 
 $\widetilde \varphi\colon \pi^{-1}(U)\to S^* \RR^n$  the induced diffeomorphism and   write $p=\varphi^{-1}( x)$. Using \eqref{eq:restriction_bundle_map} we compute 
 \begin{align*}
  \int_{ N^*(P)\cap \pi^{-1}(F)} \omega & = \int_{ N^*(\varphi(P\cap U)) \cap \pi^{-1}(\varphi(F\cap U)) } (\widetilde\varphi^{-1})^*\omega\\
  & = \int_{\varphi(F\cap U)} \int_{ ((-\infty,0]^{d-k}\times 0_{\RR^{k}}\times \RR^{n-d})/\sim} r_{ x}(j^*(\widetilde\varphi^{-1})^*\omega)\\  
    & = \int_{\varphi(F\cap U)} (\varphi^{-1})^* \int_{  ((-\infty,0]^{d-k}\times 0_{\RR^{k}}\times \RR^{n-d})/\sim} (\widetilde\varphi^{-1})^* r_{p}(j^*\omega)\\
        & = \int_{\varphi(F\cap U)} (\varphi^{-1})^* \int_{ N(P)\cap \pi^{-1}(p)} r_{p}(j^*\omega)\\
        & = \int_F  \int_{N(P)\cap\pi^{-1}(p)} r_p(j^*\omega)
 \end{align*}

\end{proof}

\subsection{Riemannian geometry}

  \subsubsection{Riemannian submanifolds}

 Let $M\subset \overline{M}$ be an embedded submanifold of $\overline M$ with 
 induced Riemannian metric. We denote the Levi-Civita connections on $M, \overline M$ by $\nabla,\overline \nabla$ and by $h$ the
 second fundamental form of $M\subset \overline M$. If $X,Y$ are vector fields on $M$ extended 
 arbitrarily to a neighborhood of $M$ in  $\overline{M}$, then  the decomposition into normal and tangential parts
 \begin{equation}\label{eq:gauss} \overline \nabla_X Y = \nabla_XY + h(X,Y)\end{equation}
 along $M$ is called the  Gauss formula. This decomposition adapted to vector fields tangent to $M$ along  curves $\gamma$ in $M$ yields
 \begin{equation}\label{eq:acceleration}\overline\nabla _{\gamma'} Y =\nabla _{\gamma'} Y + h(\gamma', Y),\end{equation}
 see \cite{oneill}*{p. 102}. 
 In particular, if  $\exp^{\overline M} (T_p M\cap B(0,\varepsilon))\subset M$
 for some $\varepsilon>0$, where $\exp^{\overline M}$ denotes the exponential map of $\overline M$, the second fundamental form of $M$ vanishes at $p$. 
 
 If $X,\eta$ are vector fields on $M$ with $X$ always tangent to $M$ and $\eta$ always normal to $M$ extended arbitrarily to a neighborhood of $M$ in  $\overline M$,
 then there exists an analogous decomposition into a normal and tangential parts
 $$\overline \nabla_X \eta = S_\eta(X) + \nabla^\perp_X\eta,$$
 where $S_\eta\colon T_pM\to T_p M$ is the shape operator $\langle S_\eta(X),Y\rangle =\langle h(X,Y),\eta\rangle$ and $\nabla^\perp$ the normal connection.
This decomposition adapted to vector fields normal to $M$ along  curves $\gamma$ in $M$ yields
 \begin{equation}\label{eq:acceleration_normal}\overline\nabla _{\gamma'} \eta = S_\eta(\gamma') + \nabla^\perp _{\gamma'} \eta ,\end{equation}

 \subsubsection{Horizontal lifts}
 Let $M$ be a Riemannian manifold and denote by $\pi\colon SM\to M$ the  sphere bundle of $M$. The horizontal lift of a tangent vector $X\in T_pM$ to a point 
 $u$ of the sphere bundle with $\pi(u)=p$ is by definition  
 $$\widetilde X =  Y'(0)\in T_u SM,$$
 where $Y$ is a vector field of unit length along a curve $\gamma$ in $M$ such that 
 \begin{align}\label{eq:horizontal_lift}\begin{split} \gamma(0)& =p,\quad  \gamma' (0) = X,\\
  Y(0)&=u,\quad   \left.\nabla_{\gamma'} Y\right|_0  = 0.
  \end{split}
 \end{align}
Similarly, if $M\subset \overline M$ is an embedded submanifold of $\overline M$ with induced Riemannian metric and 
$\pi \colon \nu M\to M$ denotes the unit normal bundle to $M$, then the horizontal lift of a tangent vector $X\in T_pM$ to a point 
 $\xi$ of the normal bundle $\pi(\xi)=p$ is by definition 
  $$\widetilde X =   Y'(0)\in T_\xi \nu M,$$
 where $Y$ is a unit length vector  field always normal to $M$ along a curve $\gamma$ in $M$ such that  
 \begin{align}\label{eq:horizontal_lift_normal} \begin{split} \gamma(0)& =p,\quad  \gamma' (0) = X,\\
  Y(0)&=\xi,\quad   \left.\nabla_{\gamma'}^\perp Y\right|_0  = 0,
  \end{split}
 \end{align}
 where $\nabla^\perp$ denotes the normal connection. 
 
 \begin{lemma} \label{lemma:lifts_S} Let $M\subset \overline{M}$ be an embedded submanifold of $\overline M$ with 
 induced Riemannian metric. 
  If  $\exp^{\overline M} (T_p M\cap B(0,\varepsilon))\subset M$
 for some $\varepsilon>0$, then the horizontal lift 
 of $X\in T_p M$ to a direction   $u$ in $S\overline M$ normal to $M$ coincides with 
 the horizontal lift of $X$ to  $u$ in the unit normal bundle of $M\subset \overline M$. 
 \end{lemma}
\begin{proof}
 This is an immediate consequence of  \eqref{eq:acceleration_normal}.
\end{proof}

 The horizontal lift of a tangent vector $X\in T_pM$ to a point 
 $\xi$ of the tangent bundle $TM$ with $\pi(\xi)=p$ is 
 $$\widetilde X =   Y'(0)\in T_\xi TM,$$
 where $Y$ is a vector field along a curve $\gamma$ in $M$ such that \eqref{eq:horizontal_lift} holds. 
 The horizontal lift to normal bundle $T M^\perp$ of $M\subset \overline M$ is defined similarly, but now the conditions on $Y$ and $\gamma$ are 
 \eqref{eq:horizontal_lift_normal}. As before we have

 \begin{lemma}\label{lemma:lifts_T}Let $M\subset \overline{M}$ be an embedded submanifold of $\overline M$ with 
 induced Riemannian metric.  If 
   $\exp^{\overline M} (T_p M\cap B(0,\varepsilon))\subset M$
 for some $\varepsilon>0$, then the horizontal lift    
 of $X\in T_p M$ to $\xi $ in $ T\overline M$ normal to $M$ coincides with 
 the horizontal lift of $X$ to  $\xi$ in the  normal bundle of $M\subset \overline{M}$.
 \end{lemma}

\subsection{Representation theory}

\subsubsection{Representations of the general and special linear groups} \label{subsec:gln}

We first recall that the isomorphism classes of irreducible  representations of  $GL(n,\CC)$
may be  param\-etrized by nonincreasing integer sequences $\lambda_1\geq\ldots\geq \lambda_n$.
A nonincreasing sequence   $\lambda=(\lambda_1,\ldots, \lambda_m)$ of nonnegative integers  is a called a partition
with at most $m$ parts. The transpose partition $\lambda'$ is defined by $\lambda'_i = |\{\lambda_j\colon \lambda_j\geq i\}|$.
The irreducible  representations of  $SL(n,\CC)$ are parametrized by partitions with at most $n-1$ parts,  see \cite{fulton_harris}*{\S 15.5}.
If $\lambda, \mu$ are nonincreasing integer sequences and there exists an integer $k$ with  $\lambda_i=\mu_i+k$ for $i=1,\ldots, n$, then  the restrictions 
of the corresponding representations to $SL(n,\CC)$ are isomorphic. More precisely, the restriction of the irreducible $GL(n,\CC)$-representation $\lambda_1\geq\ldots\geq \lambda_n$ to $SL(n,\CC)$ is the representation $\mu_1\geq\cdots\geq \mu_{n-1}$ with 
$\mu_i=\lambda_i-\lambda_n$ for $i=1,\ldots,n-1$. 
For example, the representations $\largewedge^i V$ for $i=1,\ldots, n$, where $V=\CC^n$ is the standard representation of $GL(n,\CC)$ correspond to 
$$\lambda_1=\cdots=\lambda_i=1 \qquad \text{and}\qquad \lambda_{i+1}=\cdots =\lambda_n=0.$$

Sometimes it is more convenient to parametrize the irreducible representations of $SL(n,\CC)$ in terms of their highest weights
$$m_1\varpi_1+ \cdots + m_{n-1} \varpi_{n-1}$$
where $m_1,\ldots,m_{n-1}$ are nonnegative integers and the $\varpi_i$ are the fundamental weights of $SL(n,\CC)$, i.e., 
the highest weights of the irreducible representations $\largewedge^i V$. This parametrization is related to the former by
$\lambda_i=m_i+\cdots+m_{n-1}$ for $ i=1,\ldots, n-1$.

\begin{lemma} \label{lemma:tensor decomp} Let $1\leq k\leq n-1$. 
 The $SL(n,\CC)$-representations 
 $$\largewedge^k V^* \otimes \largewedge^{n-k} V \qquad \text{and} \qquad \largewedge^{k-1} V^* \otimes \largewedge^{n-k-1} V$$
 differ by precisely one irreducible submodule of highest weight $2\varpi_{n-k}$. Moreover, only this  representation
 can occur as a  common irreducible submodule of 
  $$\Sym^2(\largewedge^k V^*) \qquad \text{and} \qquad \Sym^k(\Sym^2 V^*)$$
\end{lemma}
\begin{proof} Since 
\begin{equation}\label{eq:duals}\largewedge^k V^*\cong  \largewedge^{n-k} V
\end{equation}
 as $SL(n,\CC)$-representations,  to prove the first  assertion it is enough  to show that  
$\largewedge^i V\otimes \largewedge^i V$ and $\largewedge^{i+1} V\otimes \largewedge^{i-1} V$ differ for each $i=1,\ldots, n-1$ 
by precisely one irreducible subrepresentation 
of highest weight $2\varpi_i$. But this follows immediately from the Littlewood-Richardson rule described in Section~\ref{sec:littlewood} below.

To prove the second assertion, observe that 
$$V^*\otimes V^*\cong \largewedge^2 V^* \oplus \Sym^2 V^*.$$
Hence using again \eqref{eq:duals} and the Littlewood-Richardson rule, we find that 
$\Sym^2 V^*$ is the irreducible representation with highest weight $2\varpi_{n-1}$. Next, since 
$$\Sym^k(\Sym^2 V^*)\subset (\Sym^2 V^*)^{\otimes k},$$ the Littlewood-Richardson rule implies that 
all the integer partitions $\lambda$ with at most $n-1$ parts that correspond to irreducible 
subrepresentations $\Sym^k(\Sym^2 V^*)$ must  satisfy $\lambda'_1,\lambda'_2\geq n-k$. But 
$$\Sym^2(\largewedge^k V^*)\subset (\largewedge^k V^*)^{\otimes 2}$$ and so it follows again from \eqref{eq:duals} and the Littlewood-Richardson rule that 
all the integer partitions $\lambda$ with at most $n-1$ parts that correspond to irreducible 
subrepresentations $\Sym^2(\largewedge^k V^*)$ must  satisfy $\lambda'_1+\lambda'_2\leq 2n-2k$. 
Thus only the irreducible representation with highest weight $2\varpi_{n-k}$ can occur as a submodule of both    $\Sym^2(\largewedge^k V^*)$  and $\Sym^k(\Sym^2 V^*)$.
\end{proof}

By the Weyl dimension formula for $SL(n,\CC)$, 
the dimension of the irreducible representation $\Gamma_\lambda=\Gamma_{\lambda_1,\ldots,\lambda_{n-1}}$  is given by
\begin{equation}\label{eq:weyl_dim} \dim (\Gamma_\lambda) = \prod_{1\leq i<j \leq n } \frac{\lambda_i-\lambda_j+j-i}{j-i},\end{equation}
where $\lambda_n=0$, see, e.g., \cite{fulton_harris}*{\S 15.3}.

\subsubsection{Representations of the orthogonal groups}
Next we recall the parametrizations of the complex orthogonal groups.  
The irreducible representations of $O(n,\CC)$ are parametrized by partitions with at most $n$ parts satisfying
$$ \lambda'_1+\lambda_2'\leq n,$$
see \cite{fulton_harris}*{Theorem 19.19}.
The description of the representations of the special orthogonal group is sensitive to the parity of the dimension $n$.
The irreducible representations of $SO(2m+1,\CC)$ are parametrized by integer sequences $(\lambda_1, \ldots, \lambda_m)$ with \begin{equation}\label{eq:parametr_so_odd}\lambda_1\geq\ldots\geq \lambda_m\geq 0;\end{equation}
the irreducible representations of $SO(2m,\CC)$ are parametrized by integer sequences $(\lambda_1, \ldots, \lambda_m)$ with \begin{equation}\label{eq:parametr_so_even}\lambda_1\geq\ldots\geq \lambda_m\quad \text{with} \quad\lambda_{m-1}\geq |\lambda_m|.\end{equation}

Two partitions $\lambda$ and $\mu$ with at most $n$ parts are called associated if $\lambda'_1+\mu'_1=n$ and $\lambda'_i=\mu_i'$ for $i>1$. Representations of $O(n,\CC)$ corresponding to 
associated partitions restrict to isomorphic representations of $SO(n,\CC)$.  Note that at least one of  each pair of 
associated partitions will have at most $\frac 12 n$ parts.  If $n=2m+1$ is odd, then these restrictions are irreducible and the restricted representation is given by \eqref{eq:parametr_so_odd}.
The same holds if $n=2m$ is even, and $\lambda_m=0$. But if $\lambda_m>0$, then the restriction is the sum of two irreducible $SO(n,\CC)$-representations
$$\lambda_1\geq \cdots \geq \lambda_{m-1}\geq \lambda_m \quad \text{and}\quad \lambda_1\geq \cdots \geq \lambda_{m-1}\geq -\lambda_m.$$
For example,  the representations $\largewedge^i V$ for $i=1,\ldots, n$, where $V=\CC^n$ is the standard representation of $O(n,\CC)$ correspond to 
$$\lambda_1=\cdots=\lambda_i=1 \qquad \text{and}\qquad \lambda_{i+1}=\cdots =\lambda_n=0.$$
In particular,  when restricted to $SO(n,\CC)$ these representations are irreducible when $n$ is odd; the same holds if  $n$ is even and $i<n$, but  $\largewedge^n V$ is not irreducible: the eigenspaces of the Hodge star operator are the irreducible subrepresentations.  

\subsubsection{Littlewood-Richardson coefficients and the Littlewood restriction rule} \label{sec:littlewood}
If $\Gamma_\lambda$ and $\Gamma_\mu$ are irreducible representations of $GL(n,\CC)$ corresponding to integer partitions $\lambda, \mu$ with at most $n$ parts, then their tensor product
decomposes into irreducible components as 
$$\Gamma_\lambda \otimes \Gamma_\mu = \bigoplus N_{\lambda\mu\nu}\Gamma_\nu,$$
where the sum is over integer partitions $\nu$ with at most $n$ parts and the numbers $N_{\lambda\mu\nu}$ are the Littlewood-Richardson coefficients. These are determined by the Littlewood-Richardson rule  (see, e.g., \cite{fulton_harris}*{Appendix A}): 
To each partition $\lambda$ is associated a Young diagram
$$\yng(3,3,2,1,1)$$
with $\lambda_i$ boxes in the $i$th row and the rows of boxes lined up on the left. 
A $\mu$-expansion of $\lambda$ is a partition obtained from $\lambda$ by  first adding $\mu_1$ boxes and putting the integer $1$ in each new box; then adding $\mu_2$ boxes with a $2$, continuing until finally $\mu_k$ boxes with the integer $k$ in each box are added. In each step, the boxes are added such that no two are in 
the same column. The $\mu$-expansion is called strict if, when the integers in the  boxes are listed from right to left, starting with the top row and working down, one looks at the first $t$ entries in this list (for any $t$ between $1$ and $\mu_1+\cdots+ \mu_k$), each integer $p$ between $1$ and $k-1$ occurs at least as many times 
as the next integer $p+1$. The Littlewood-Richardson coefficient is the number of ways the Young diagram $\lambda$ can be expanded  to the Young diagram $\nu$  by a strict $\mu$-expansion.

For the $\lambda$ in the so-called stable range, the decomposition into irreducible components of the restriction from $GL(n,\CC)$ to $O(n,\CC)$ is  well-known. 
 By the Littlewood restriction rule, see, e.g.,  \cite{howe_etal}, if  $\lambda$ is a partition with at most $\lfloor n/2 \rfloor$ parts, then 
 $$\operatorname{Res}^{GL(n,\CC)}_{O(n,\CC)} (\Gamma_\lambda )= \bigoplus N_{\lambda\mu} \Gamma_\mu$$
 where the sum is over all $\mu=(\mu_1\geq \cdots \geq \mu_n\geq 0)$ with $\mu_1'+\mu_2'\leq n$, where
 $$N_{\lambda\mu}= \sum N_{\delta\mu\lambda}$$
 with $N_{\delta\lambda\mu}$ the Littlewood-Richardson coefficient, where the sum is over all partitions $\delta$ with all $\delta_i$ even.

  It is not difficult to adapt the above to representations of the special linear group. Given representions of 
 $SL(n,\CC)$ corresponding to integer partitions $\lambda, \mu$  with at most $n-1$ parts, consider the representations of $GL(n,\CC)$ with $\lambda_1\geq\cdots \geq \lambda_{n-1}\geq \lambda_n=0$ and 
  $\mu_1\geq\cdots \geq \mu_{n-1}\geq \mu_n=0$. Decomposing their tensor product using the Littlewood-Richardson rule and restriciting back to $SL(n,\CC)$ gives the result for the special linear group. The same works for the Littlewood restriction rule.

 \begin{lemma}\label{lemma:littlewood}
  If $k\leq n/2$ and 
  $$\lambda_1=\cdots = \lambda_k=2 \quad \text{and}\quad \lambda_{k+1}=\cdots=\lambda_{n-1}=0,$$
  then 
  $$\operatorname{Res}^{SL(n,\CC)}_{SO(n,\CC)} (\Gamma_\lambda )= \bigoplus \Gamma_\mu$$
  where the sum is over all sequences $\mu=(\mu_1,\ldots,\mu_k,0,\ldots, 0)$ of even integers  satisfying $|\mu_1|\leq2 $ in addition to \eqref{eq:parametr_so_odd} and \eqref{eq:parametr_so_even}.
 \end{lemma}
\begin{proof}
Let $\delta$ be a partition into at most $n$ parts consisting only of even integers. 
Suppose that $\lambda$ is a strict $\mu$-expansion of $\delta$. Since clearly $\lambda_i\geq \delta_i$ we must have 
$$\delta_1=\cdots =\delta_i=2 \quad\text{and}\quad \delta_{i+1}=\cdots=\delta_n=0$$
for some $i\leq k$. Since no two boxes can be added to one column of $\delta$, each row of $\mu$ contains at most $2$ boxes. Now $ \lambda$ is a strict $\mu$-expansion,
so  each nonzero row of $\mu$ consists of precisely two boxes. This proves that 
if $\delta$ is even, then 
$N_{\delta\mu\lambda}=1$
if and only if $\delta'+\mu'=\lambda'$. Therefore the  Littlewood restriction rule implies 
  $$\operatorname{Res}^{SL(n,\CC)}_{O(n,\CC)} (\Gamma_\lambda )= \bigoplus \Gamma_\mu$$
  where the sum is over all partitions $\mu$ with 
$$\mu_1=\cdots =\mu_i=2 \quad\text{and}\quad \mu_{i+1}=\cdots=\mu_n=0.$$
for some $i\leq k$. 

\end{proof}

\section{Angular curvature measures} 
\subsection{Translation-invariant curvature measures} Let $\Curv(\RR^n)\subset \calC(\RR^n)$ denote the subspace of translation-invariant curvature measures.
Recall that $\Curv(\RR^n)$ is graded by degree of homogeneity, $$\Curv(\RR^n)=\bigoplus_{k=0}^n \Curv_k(\RR^n).$$ 
Following \cite{bfs}*{Section 2.2.} we associate to each   $\Psi\in \Curv(\RR^n)$  
a function $c_{\Psi}$ on closed convex cones in $\RR^n$ 
$$c_\Psi (C)  = \lim_{r\to 0} \frac{1}{\omega_k r^k} \Psi(C, L \cap B(0,r)),
$$
where $L$ is the lineality space of  $C$, $k=\dim L$, and $B(0,r)$ denotes the closed ball of radius $r$ centered at $0$.
An immediate consequence of the finite additivity of the normal cycle is that $c_\Psi$ is finitely additive on the set of 
cones having the same lineality space.

If $\Psi$ is given  by  differential forms $(\psi,\theta)\in \Omega^{n-1}(S\RR^n)\oplus \Omega^n(\RR^n)$ and $k=\dim L(C)\leq n-1$, then a direct computation shows that 
\begin{equation}\label{eq:c forms} c_\Psi (C)= \int_{C^\circ \cap S^{n-1}} \vec L \lrcorner   \psi  \end{equation}                                                                                                                 
where $\vec L= u_1\wedge \cdots \wedge  u_k$ for some suitably oriented orthonormal basis of $L$.

The following characterization of angularity will be important for us. 
\begin{lemma}\label{lemma:char_angularity}
 A curvature measure $\Psi\in\Curv(\RR^n)$ is angular if and only if 
 there is a  function $f\colon \coprod_k \Grass_k(\RR^n)\to \RR$ such that for every closed convex cone $C\subset \RR^n$ with lineality space $L(C)=L$  
$$c_\Psi(C)= f(L)\gamma(C),$$
 where $\gamma(C)$ is the external angle of $C$, see \eqref{eq:exterior_angle}.
\end{lemma}
\begin{proof}
The Lemma is an immediate consequence of the following fact:  for every polytope $P\subset \RR^n$
\begin{equation}\label{eq:c Psi}\Psi(P, E ) = \sum_{k=0}^n  \sum_{F\in\calF_k(P)} c_\Psi (T_FP ) \vol_k(F\cap E),\end{equation}
where $\calF_k(P)$ denotes the set of $k$-faces of $P$.  
To see the latter, just observe that if $\Psi$ is given  by  differential forms $(\psi,\theta)$ 
$$\Psi(P, E)  =  \sum_{k=0}^{n-1}  \sum_{F\in\calF_k(P)} \int_{(F\cap E) \times (T_F P) ^\circ \cap S^{n-1}}  \psi + \int_{P\cap E}\theta.$$
If the forms $(\psi,\theta)$ are translation-invariant then comparing with \eqref{eq:c forms} yields \eqref{eq:c Psi}. 
\end{proof}

\label{subsec:ccc}

If a differential form $\omega\in \Omega^{n-1}(S\RR^n)$ is  extended to $T\RR^n\supset S\RR^n$ then by Stokes' theorem
$$\int_{N(P)} \omega= \int_{N_1(P)} d\omega$$
where 
$$N_1(P)= m_*((0,1]\times N(P)) + [P]\times_\pi [0]$$
with $m\colon (0,\infty)\times S\RR^n \to T\RR^n$, $m(u,t)= tu$, is the normal disc current, see \cites{bf_convolution,hig}. 
Recall from the introduction that a curvature measure $\Psi\in \Curv(\RR^n)$ is said to be a constant coefficient curvature measure
 if there exists a constant coefficient form $\omega\in \largewedge^n(\RR^n\oplus \RR^n)^*\subset \Omega^n(\RR^n\oplus \RR^n)$ such that
 \begin{equation}\label{eq:ccc}\Psi(P,U)= \int_{N_1(P)\cap \pi^{-1}(U)} \omega\end{equation}
  for all compact submanifolds with corners $P\subset \RR^n$  and all Borel sets $U\subset \RR^n$.

 A direct computation shows the following
 \begin{lemma}[\cite{bfs}*{Lemma 2.30}] Every constant coefficient curvature measure is angular.  
 \end{lemma}

\subsection{Curvature measures on Riemannian manifolds} \label{sec:transfer} Let $M$ be a  Riemannian manifold. The paper \cite{bfs} describes a linear  isomorphism $\tau$ between $\calC(M)$ the space of smooth curvature measures on $M$
and the smooth sections of the bundle   $\Curv(TM)=\coprod_{p\in M} \Curv(T_pM)$ of
translation-invariant smooth cuvature measures on the tangent spaces of $M$. Let us recall its construction. The tangent spaces of $SM$ 
 split into subspaces of horizontal and vertical tangent vectors
 $T_\xi SM \cong  H\oplus V$, $\xi \in S_p M$. 
 Since every $X_p\in T_pM$ has a unique  horizontal lift, $T_pM$ may be canonically identified with $H$ and so 
 \begin{equation}\label{eq_tangent_space_iso} T_\xi SM \cong T_pM \oplus \langle \xi\rangle ^\perp.\end{equation}
 This identification induces an  isomorphism  $\largewedge ^k T_\xi SM  \cong \bigoplus_{i=0}^k
 \largewedge^i T_pM\otimes \largewedge^{k-i} T_\xi S_pM$ for each $\xi\in S_pM$. Consequently, 
 fixing $p\in M$ and letting $\xi\in S_p M$ vary, these identifications yield a map 
 $$ \Omega^k(SM)|_{S_pM} \to \Omega^k(ST_pM)^{tr}$$
 to the space of translation-invariant $k$-forms on the sphere bundle of $T_pM$. Letting $p\in M$ vary we obtain
 an isomorphism
 $\tau\colon \Omega^k(SM) \to \Gamma( \Omega^{k}(STM)^{tr})$
  between smooth $k$-forms on $SM$ and smooth sections of the bundle $ \Omega^{k}(STM)^{tr}$ of smooth translation-invariant $k$-forms
  on $ST_pM$.
  
 Let $\bar \alpha$ denote the canonical contact $1$-form on $ST_pM$. It was shown in \cite{bfs}*{Proposition 2.5}
  that $\alpha,d\alpha$ correspond to $\bar \alpha, d\bar\alpha$ under the above identification. 
  Thus $\tau$ descends to a well-defined isomorphism 
  $$\tau\colon \calC(M)\to \Curv(TM).$$
We will also write $\tau_p \Phi$ instead of $\tau(\Phi)_p$. 
  
Recall from the introduction  that  a smooth curvature measure $\Phi$ on $M$ is called angular if $\tau_p\Phi \in\Curv(T_pM)$ is angular for all $p\in M$. 
The subspace of angular curvature measures is denoted by $\calA(M)$.  The curvature measure $\Phi$ is called angular at $p\in M$ if $\tau_p\Phi $  is angular.

\begin{lemma}
 A smooth curvature measure $\Psi\in\calC(M)$ is angular at $p$ if and only if there exists a  function $f\colon \coprod_k \Grass_k(T_p M)\to \RR$  such that
for every closed convex cone $C$ in $T_pM$  with lineality space $L(C)=L$
 $$c_{\tau_p\Psi}(C)=f(L) \gamma(C).$$ 
\end{lemma}
\begin{proof}
 This is an immediate consequence of Lemma~\ref{lemma:char_angularity}.
\end{proof}

\section{Proof of Theorem~\ref{thm:ccc}}

\begin{lemma}\label{lemma:dim ccc} For $0\leq k<n-1$ the dimension of the space of $k$-homogeneous constant coefficient curvature measures on $\RR^n$ is 
 \begin{equation}
  \label{eq:dim ccc} \frac{1}{n-k+1}\binom{n}{k}\binom{n+1}{k+1}
 \end{equation}
\end{lemma}
 \begin{proof}

Let $\Psi$ be a $k$-homogeneous constant coefficient curvature measure, say 
$$\Psi(P,U)= \int_{N_1(P)\cap\pi^{-1}(U)} \omega,$$
where $\omega\in  \largewedge^k\RR^{n*}\otimes \largewedge^{n-k} \RR^{n*}$. 
If $P\subset \RR^n$ is a polytope, then 
$$\Psi(P,U)=\sum f(\overline F)\gamma(F,P) \vol_k(F\cap U)$$
where the sum is over all $k$-dimensional faces $F$ of $P$ and the number $f(\overline F)$ depends only on the unique $k$-dimensional linear subspace  $\overline F$
parallel to the affine hull of $F$. Moreover, as functions on $\Grass_k(\RR^n)$,  $f$ and 
$$E\mapsto  \int_{B_E\times B_{E^\perp}} \omega,$$
where $B_E$ is the $k$-dimensional Euclidean unit ball in $E$, differ only by a constant nonzero multiple.
Hence  $\Psi=0$ if and only if $\omega|_{E\oplus E^\perp}=0$ for all
$k$-dimensional linear subspaces $E\subset \RR^n$. Let $U$ denote the subspace
of all $\omega\in  \largewedge^k\RR^{n*}\otimes \largewedge^{n-k} \RR^{n*}$ with this property. To prove the Lemma it suffices to determine the 
dimension of $U$. 

Let $V$ be an $n$-dimensional real vector space and let $E^\circ\subset V^*$ denote the annihilator of a linear subspace $E\subset V$. 
We introduce another subspace $W$, namely the  subspace of all $\omega\in \largewedge^k V^* \otimes \largewedge^{n-k} V$ with the property that $\omega|_{E\oplus E^\circ}=0$ 
for all $k$-dimensional subspaces $E$ of $V$. Observe that a choice of Euclidean inner product on $V$ yields a linear isomorphism from $W$ onto $U$.

Note that $W$ is an invariant subspace for the natural action of $SL(n,\RR)$ on $\largewedge^k V^* \otimes \largewedge^{n-k} V$;
its complexification $W^\CC$ consists of all complex-valued $n$-covectors on $V\oplus V^*$ of bidegree $(k,n-k)$ vanishing on 
$E\oplus E^\circ$ for all $k$-dimensional subspaces $E\subset V$. 
If $\omega$ is a multiple of the standard symplectic form on $V\oplus V^*$, then $\omega \in W^\CC$. Since moreover multiplication by the symplectic form is an isomorphism $\largewedge^{n-2}(V^*\oplus V)\to \largewedge^{n}(V^*\oplus V)$, we conclude that $W^\CC$ contains an 
$SL(n,\CC)$-submodule isomorphic to 
\begin{equation}\label{eq:submodule}\left(\largewedge^{k-1} V^* \otimes \largewedge^{n-k-1} V \right)^\CC=  \largewedge^{k-1} (V^\CC)^* \otimes \largewedge^{n-k-1} (V^\CC) \end{equation}
Here  we have used that complexification commutes with taking exterior powers and duals. As $SL(n,\CC)$-modules,  \eqref{eq:submodule} and 
$\largewedge^{k} (V^\CC)^* \otimes \largewedge^{n-k} (V^\CC)$ differ according to Lemma~\ref{lemma:tensor decomp} by precisely one irreducible component.
This proves that $W^\CC$ coincides with \eqref{eq:submodule}. Hence
$$\operatorname{codim}   U = \codim W =\binom{n}{k}^2 - \binom{n}{k-1}\binom{n}{k+1}= \frac{1}{n-k+1}\binom{n}{k}\binom{n+1}{k+1},$$
as required.
\end{proof}

Let $V$ be an $n$-dimensional real vector space and let  $ \Grass_k(V)$ denote the Grassmannian of $k$-dimensional linear subspaces of $V$.
Recall that the Pl\"ucker embedding, $\psi\colon \Grass_k(V)\to \PP(\largewedge^k V)$, is defined by 
$$\psi(E)=[e_1\wedge \cdots\wedge e_k]$$
where $e_1,\ldots,e_k$ is a basis of $E$. 
We will need the following well-known description of the restriction of $2$-homogeneous polynomials to the image of the Pl\"ucker embedding, see, e.g., \cite{fulton_harris}*{\S 15.4}. For the convenience of the reader
we include the short proof. 

\begin{lemma}\label{lemma:plucker} Under the natural action of $SL(n,\CC)$ the vector space 
$$\Sym^2(\largewedge^k V^*)^\CC/ \{ p\equiv 0 \text{ on } \im \psi\}^\CC$$
is irreducible with highest weight $2\varpi_{n-k}$. In particular, its dimension is
\begin{equation}\label{eq:plucker dim} \frac{1}{n-k+1}\binom{n}{k}\binom{n+1}{k+1}.\end{equation}
\end{lemma}
\begin{proof}
Note that each $p\in \Sym^2(\largewedge^k V^*)$ defines a quadratic polynomial  
$$(v_1,\ldots,v_k)\mapsto p(v_1\wedge \cdots \wedge v_k)$$
on $V^k$ that is  $2$-homogeneous in each argument and symmetric. Thus we have a natural map 
$$f\colon \Sym^2(\largewedge^k V^*) \to \Sym^k(\Sym^2 V^*)$$
Obviously, $\ker f =\{ p\equiv 0 \text{ on } \im \psi\}$. 

 Since complexification commutes with taking  symmetric powers, exterior powers, and duals, the complexification 
of $f$ is a homomorphism of $SL(n,\CC)$-modules 
$$f^\CC\colon \Sym^2(\largewedge^k (V^\CC)^*) \to \Sym^k(\Sym^2 (V^\CC)^*).$$
Note that $\ker f^\CC =(\ker f)^\CC$. 
Lemma~\ref{lemma:tensor decomp} implies that   $\Sym^2(\largewedge^k (V^\CC)^*)$ and  $\Sym^k(\Sym^2 (V^\CC)^*)$ have precisely one common $SL(n,\CC)$-submodule, namely the one with highest weight 
$2\varpi_{n-k}$; the Weyl dimension formula \eqref{eq:weyl_dim} yields the  desired conclusion.
\end{proof}

Recall from the introduction that  $\widetilde\Grass_k(\RR^n)$ denotes the oriented Grassmannian and that we call a function on $\widetilde\Grass_k(\RR^n)$ even if it is invariant under change of orientation. 
The oriented Grassmannian smoothly embeds into $\largewedge^k\RR^n$ as $E\mapsto \vec E$, where $\vec E=e_1\wedge \cdots \wedge e_k$ 
for some positively oriented orthonormal basis of $E$. We also call this map the Pl\"ucker embedding. Observe that it respects the natural actions of $SO(n)$ on 
$\widetilde\Grass_k(\RR^n)$ and $\largewedge^k\RR^n$.  Forgetting about orientations gives a map
$\widetilde\Grass_k(\RR^n)\to \Grass_k(\RR^n)$. 

\begin{lemma}\label{lemma:highest_weights_restrictions}
 The subspace of $C(\widetilde\Grass_k(\RR^n))$ of restrictions of $2$-homogeneous polynomials on $\largewedge^k \RR^n$  to $\widetilde \Grass_k(\RR^n)$ decomposes under the natural action of 
 $SO(n)$  precisely into the  irreducible representation with highest weights $(2 m_1, 2 m_2,\ldots, 2 m_{\lfloor n/2\rfloor})$ satisfying 
 $|m_{1}|\leq 1$ and $m_i=0$ for $i>\min(k,n-k)$.  
 
\end{lemma}
\begin{proof} By taking orthogonal complements, we may assume without loss of generality that $k\geq \lfloor n/2\rfloor$. 
 The space of restrictions of $2$-homogeneous polynomials  to $\widetilde \Grass_k(\RR^n)$ is $SO(n)$-equivariantly isomorphic 
 to $$\Sym^2(\largewedge^k V^*)/ \{ p\equiv 0 \text{ on } \im (\psi\colon \Grass_k(V)\to \PP(\largewedge^k V))\},$$
 where $V=\RR^n$. By Lemma~\ref{lemma:plucker}, the complexification of this space is precisely the restriction of the $SL(n,\CC)$-representation 
 with highest weight $2\varpi_{n-k}$ to $SO(n)$.   An application of Lemma~\ref{lemma:littlewood} completes the proof of the Lemma.
\end{proof}

The assumption $k<n-1$ in Theorem~\ref{thm:ccc} is needed for the following

\begin{proposition}\label{prop:polynomial_restriction}
 Let $1\leq k <n-1$ and $f$ be an even continuous function on $\widetilde{\Grass}_k(\RR^n)$. Suppose that for every nonzero $v\in \RR^n$ there exists a 
 $2$-homogeneous polynomial $q$ on $\largewedge^k v^\perp$ such that $f=q$ on $\widetilde{\Grass}_k(v^\perp)\subset \largewedge^k v^\perp$. Then there exists  
  a globally defined $2$-homogeneous polynomial $q$ on $\largewedge^k \RR^n$ such that 
 $f=q$ on $\widetilde{\Grass}_k(\RR^n)\subset \largewedge^k \RR^n$.
\end{proposition}

\begin{proof}

 Let $W\subset C(\widetilde{\Grass}_k(\RR^n))$ be the subspace of all even continuous functions $f$ on $\widetilde{\Grass}_k(\RR^n)$ with the property that for every nonzero $v\in \RR^n$ there exists a 
 $2$-homogeneous polynomial $q$ on $\largewedge^k v^\perp$ such that $f=q$ on $\widetilde{\Grass}_k(v^\perp)\subset \largewedge^k v^\perp$.
 We have to show that $W$ coincides with the space of restrictions of $2$-homogeneous polynomials on  $\largewedge^k \RR^n$ to  $\widetilde{\Grass}_k(\RR^n)$. 
 
 To begin with, the subspace $W$ is invariant under the natural action of $SO(n)$ and hence decomposes into certain irreducible subrepresentations. 
 The highest weights occurring in the representation $C(\Grass_k(\RR^n))$ are well
known. According to Strichartz \cite{strichartz} they all have multiplicity $1$ and are of the form 
$(2m_1,\ldots, 2m_{k'},0\ldots, 0)$, $k'=\min(k,n-k)$, 
with integers satisfying 
$$m_1\geq \ldots \geq m_{k'-1}\geq  |m_{k'}|\qquad \text{and} \qquad \text{if}\  2k'<n,\ m_k'\geq 0.$$
By Lemma~\ref{lemma:highest_weights_restrictions}, the proof will be finished if we can show that only the highest weights with
 $$ |m_1|\leq 1$$
 can occur in $W$. For this it will suffice to show that for all other highest weights the corresponding highest weight vector $f_{m_1,\ldots, m_{k'}}$ is not element of $W$. This step will be based on an 
 explicit description of the highest weight vectors given by Strichartz  \cites{strichartz} (see also \cite{alesker_hard_lefschetz}).
 
In the following we may assume without loss of generality $k\leq \lfloor n/2\rfloor$; indeed, if $k>\lfloor n/2 \rfloor$, then $f_{m_1,\ldots,m_{k'}}\circ \perp$ 
is a highest weight vector for the irreducible subrepresentation of 
highest weight $(2m_1,\ldots, 2m_{k'},0\ldots, 0)$ and we may repeat the argument given below for the case $k\leq \lfloor n/2\rfloor$ with $E^\perp$.

For any subspace $E\in \widetilde{\Grass}_k(\RR^n)$ choose an orthonormal basis $X^1,\ldots, X^k$ of $E$ and consider the corresponding $n \times k$ matrix 
 $$\begin{pmatrix}
    X^1_1  & \cdots &   X^k_1\\
    \vdots &  & \vdots \\
   X^1_{n}  &  \cdots &   X^k_{n}\\ 
   \end{pmatrix}
$$
of coordinates with respect to the standard basis $e_1,\ldots, e_{n}$ of $\RR^{n}$.
Let $X_j$ denote the $j$th row of this matrix. For $l\leq \lfloor n/2 \rfloor$ let $A(l)$ be the $l\times k$ matrix whose $j$th row is $X_{2j-1} + 
\sqrt{-1} X_{2j}$, $j=1,\ldots, l$. Note that the $l\times l$ matrix  $A(l)A(l)^t$ is independent of the choice of the orthonormal basis of $E$.

A  highest weight vector of the irreducible subrepresentation of $C(\widetilde{\Grass}_k(\RR^n))$ with highest weight $(2m_1,\ldots, 2m_k,0,\ldots, 0)$  
is given by 
\begin{enumerate}
 \item if $m_k\geq 0$,  $$ f_{m_1,\ldots, m_k}= \prod_{l=1}^{k}\det\left( A(l)A(l)^t\right)^{m_l-m_{l+1}},$$ where we set $m_{k+1}=0$;  
 \item if $m_k<0$,
 $$ f_{m_1,\ldots, m_k}= \prod_{l=1}^{k-1}\det\left( A(l)A(l)^t\right)^{m_l-|m_{l+1}|} \overline{ \det\left( A(k)A(k)^t\right)}^{|m_k|}.$$
\end{enumerate}

If $n\geq 4$  we evaluate  the highest weight vectors on $k$-planes  of the form
  $$E =\langle \cos\phi e_1 + \sin\phi  e_4\rangle\oplus \langle e_3, e_5,\ldots, e_{2k-1}\rangle\subset e_2^\perp. $$
If $n=3$ we choose $E=\langle \cos\phi e_1 + \sin\phi  e_3\rangle$. Now
$$\det(A(l)A(l)^t) =  \cos^2 \phi$$
for $l=1,\ldots,k$. Hence
$$f_{m_1,\ldots, m_k}(E) = (\cos\phi)^{2m_1},$$
which is not the restriction of a quadratic polynomial if $|m_1|>1$.   
\end{proof}

\begin{proof}[Proof of Theorem~\ref{thm:ccc}]

Since the case $k=0$ is trivial, we consider here only 
$1\leq k<n-1$.  Let $\Phi$ be an angular curvature measure given by a translation-invariant $(n-1)$-form $\varphi$ of bidegree $(k,n-k-1)$. 
By \eqref{eq:c forms} and Lemma~\ref{lemma:char_angularity} there exists an even  smooth function  $f$ on the oriented Grassmannian  $\widetilde\Grass_k(\RR^n)$ 
such that 
$$f( E)\vol_{n-k-1}(S^{n-1}\cap C) =\int_{S^{n-1}\cap C} \vec E\lrcorner \varphi$$
for every full-dimensional closed convex cone $C$ in $E^\perp$. Letting $C$ shrink to a ray we obtain
 $$f(E)\,  =  \langle \vec E\lrcorner \varphi|_{ S(E^\perp),p },  \vol_{S(E^\perp),p} \rangle, \qquad E\in \widetilde\Grass_k(\RR^n),$$
 for all $p\in S(E^\perp)$, where $\vol_{S(E^\perp)}$ denotes the volume form of the sphere $S(E^\perp)$ and $\langle \, ,\, \rangle $ is the standard inner product on $\largewedge^{n-k-1}\RR^{n*}$. 
Now fix a point $p\in S^{n-1}$. Then 
\begin{equation} \label{eq:2-homogeneous}f( E)= \langle \vec E\lrcorner \varphi_p, \vec E \lrcorner \vol_{S^{n-1},p}\rangle\end{equation}
for all $E\perp p$, since $\vec E \lrcorner\vol_{S^{n-1},p}= \vol_{S(E^\perp),p}$.  
The right-hand side of \eqref{eq:2-homogeneous} is  clearly a $2$-homogeneous polynomial on the subspace
$\largewedge^k p^\perp\subset \largewedge^k \RR^n$ restricted to the image of the Pl\"ucker embedding 
of $\widetilde\Grass_k(p^\perp)$. Proposition~\ref{prop:polynomial_restriction} implies that 
 $f$ is the restriction of a $2$-homogeneous polynomial on $\largewedge^k\RR^n$ to $\widetilde\Grass_k(\RR^n)$.  The dimension of the space of all such 
 restrictions is given by \eqref{eq:plucker dim}. Since  $f$ determines $\Phi$ as we have already seen in \eqref{eq:c Psi}, the space of angular curvature measures
 coincides by Lemma~\ref{lemma:dim ccc} with the space of constant coefficient curvature measures.
\end{proof}

\section{Proof of Theorem~\ref{thm:pullback}}

Our point of departure is  a description of the transfer map $\tau$ mentioned already in   \cite{bfs}*{p. 415}. 
 
 \begin{lemma} \label{lemma:bfs}
 Let  $M$ be a Riemannian manifold,  $\Psi$ a smooth curvature measure on $M$, and $C\subset T_pM$ a closed convex cone with lineality space $L$ such that $pr_{T_pM/L}(C)$ is simplicial. 
 For sufficiently small $\varepsilon>0$,
 $$c_{\tau_p\Psi}( C)= \lim_{r\to 0} \frac{1}{\omega_k r^k}  \Psi(\exp(C\cap  B(0,\varepsilon)),  \exp(L\cap B(0,r))  ),$$
 where $k=\dim L$.  
 \end{lemma}

 \begin{proof}
 Since the case $k=n$ is trivial, we  consider only $k<n$. Suppose that 
  $$\Psi(P,U)= \int_{N(P)\cap \pi^{-1}(U)} \omega$$
  with $\omega\in\Omega^{n-1}(SM)$ and put $F= \exp(L\cap B(0,\varepsilon))$. 
  Since $\exp(C\cap B(0,\varepsilon))$ is for sufficiently small $\varepsilon>0$ a compact smooth submanifold with corners, 
  an application of  Lemma~\ref{lemma:fiber_int}  shows 
  $$\Psi(\exp(C\cap  B(0,\varepsilon)), \exp( L\cap B(0,r))) 
  = \int_{F \cap B(0,r)}  \int_{N(\exp(C\cap  B(0,\varepsilon)))  \cap \pi^{-1}(q)} r_q(j^*\omega),
$$
where $j\colon \nu F\to SM$ is the inclusion of the unit normal bundle of $F$ into $SM$ and $r$ denotes restriction to the fibers of $\nu F$. 
Dividing by $\omega_k r^k$ and passing to the limit gives 
\begin{align*}\lim_{r\to 0} \frac{1}{\omega_k r^k}\Psi(\exp(C\cap  B(0,\varepsilon)), & \exp( L\cap B(0,r))) \\
&= \vec T_pF \lrcorner\int_{ N(\exp(C\cap  B(0,\varepsilon)))   \cap \pi^{-1}(p)}    r_p(j^*\omega) \\
&= \vec T_pF\lrcorner \int_{C^\circ\cap S^{n-1}}   r_p(j^*\omega)
\end{align*}
where $\vec{T}_pF = X_1\wedge \cdots \wedge X_k$ for some suitably oriented orthonormal basis of $T_p F$;
the last equality follows from $d\exp_0 = \id_{T_pM}$. 

By Lemma~\ref{lemma:lifts_S} the horizontal lifts  $\tilde X_1,\ldots, \tilde X_k$  of $X_1,\ldots, X_k$ 
to the unit normal bundle  $\nu F$ of $F\subset M$ coincide with the horizontal lifts to the sphere bundle $SM$. 
 Thus if $V_1,\ldots, V_{n-k-1}$ are tangent to the fiber $\nu F_p$ at $\xi\in \nu F_p$ then
\begin{align*}\vec{T}_pF\lrcorner r_p(j^*\omega)(V_1,\ldots, V_{n-k-1}) &  = \omega_\xi(V_1,\ldots, V_{n-k-1}, \tilde X_1, \ldots, \tilde X_k)\\
 & = \vec{T}_pF \lrcorner \tau_p\omega(V_1,\ldots, V_{n-k-1}),
\end{align*}
by the definition of $\tau_p\omega$. Comparing with \eqref{eq:c forms} concludes the proof.

 \end{proof}

Decomposing $T_\xi TM$ into horizontal and vertical subspaces, we get a canonical isomorphism 
$$T_\xi TM = H_\xi \oplus V_\xi \cong T_pM\oplus T_pM$$
for every $\xi \in TM$ with $\pi(\xi)= p$. Similar to the  discussion below \eqref{eq_tangent_space_iso} we see that this decomposition induces an 
isomorphism 
$$\tau \colon \Omega^k(TM)\to \Gamma( \Omega^k(TTM)^{tr})$$
between smooth $k$-forms on $TM$ and smooth sections of the bundle $\Omega^k(TTM)^{tr}$ of translation-invariant $k$-forms on $T T_p M$.

If $\omega$ is a smooth $n$-form on $TM$, we denote by $[\omega]$ the induced curvature measure on $M$,
$$[\omega](P,U)= \int_{N_1(P)\cap \pi^{-1}(U)}\omega.$$
Similarly, if $\eta$ is a smooth section of the bundle
$ \Omega^n(TTM)^{tr}$, we denote by $[\eta]$ the induced smooth section of $\Curv(TM)$. 

\begin{lemma}\label{lemma:tau}
 $\tau[\omega]=[\tau \omega]$ for every $\omega\in \Omega^n(TM)$. 
\end{lemma}

For the proof of the  Lemma we will need the following version of Lemma~\ref{lemma:fiber_int} with the normal cycle of $P$ 
replaced by the normal disc current of $P$.  
\begin{lemma}\label{lemma:fiber_int_disc} Let $M$ be a Riemannian manifold,  $P\subset M$ a compact submanifold with corners, and let
$F\subset M$ be an oriented $k$-dimensional embedded submanifold consisting only of points of $P$ of type $k$.
Let $j\colon T F^\perp\to TM$ denote the inclusion of the  normal  bundle of $F$ into the tangent bundle of $M$.
If   $\omega\in \Omega^{n}(TM)$, then 
$$ \int_{N_1(P)\cap\pi^{-1}(p)} r_p(j^*\omega)$$
depends smoothly on $p\in F$ and 
 $$\int_{N_1(P)\cap \pi^{-1}(F)} \omega = \int_F  \int_{N_1(P)\cap\pi^{-1}(p)} r_p(j^*\omega).$$
Here $r_p(j^*\omega)$ denotes the 
restriction of $j^*\omega$ to the fibers of $T F^\perp$. 
\end{lemma}
\begin{proof} Observe that the Riemannian metric gives a map $m\colon (0,\infty) \times S^*M \to TM$ such that $N_1(P)= m_*((0,1]\times N^*(P)) + [P]\times_\pi [0]$. 
The rest of the proof is now parallel to the proof of Lemma~\ref{lemma:fiber_int}.
\end{proof}

\begin{proof}[Proof of Lemma~\ref{lemma:tau}] By \eqref{eq:c Psi} it suffices to show that $c_{\tau_p[\omega]}$ equals $c_{[\tau\omega]_p}$ for each $p\in M$. Put $\Psi = [\omega]$. 
By Lemma~\ref{lemma:bfs}, we have 
for all closed convex cones $C\subset T_pM$ with $pr_{T_pM/L}(C)$ simplicial and sufficiently small $\varepsilon>0$
$$c_{\tau_p[\omega]}( C)= \lim_{r\to 0} \frac{1}{\omega_k r^k}  \Psi(\exp(C\cap  B(0,\varepsilon)), \exp(B(0,r)\cap L))$$
 Here $L=L(C)$ is the lineality space of $C$ and $k=\dim L$.  
 Repeating the computation of Lemma~\ref{lemma:bfs} with the normal disc current  $N_1(\exp(C\cap B(0,\varepsilon)))$ in place of the normal cycle
 $N(\exp(C\cap B(0,\varepsilon)))$ and using Lemma~\ref{lemma:fiber_int_disc}, we obtain further that
$$c_{\tau_p[\omega]}( C) = \vec T_pF\lrcorner \int_{C^\circ\cap D^{n}}   r_p(j^*\omega)$$
As in the proof of Lemma~\ref{lemma:bfs}, but now using Lemma~\ref{lemma:lifts_T}, we see that 
$$ \vec{T}_pF\lrcorner r_p(j^*\omega)(V_1,\ldots, V_{n-k})= \vec{T}_pF \lrcorner \tau_p\omega(V_1,\ldots, V_{n-k})$$
if $V_1,\ldots, V_{n-k}$ are tangent to the fiber $\nu F_p$.
This  finishes the proof.
\end{proof}

Note that the transfer map $\tau\colon \calC(M)\to \Curv(TM)$ induces a natural grading on $\calC(M)$. 
Neither pullback nor Alesker product respect this grading, but both are compatible with the 
filtration 
$$\calC_0(M)\supset \cdots \supset \calC_n(M)$$
where
$$\calC_i(M)=\{\text{curvature measures of degree } k\geq i\},$$ 
see \cite{sw_spheres}*{Section 3}.
We denote by $\pi_k$ the projection to the degree $k$ component of $\calC(M)$.

Next we use the characterization of angular curvature measures established in Theorem~\ref{thm:ccc} to prove the following

\begin{lemma} \label{lemma:constant_coeff_form} Let $M$ be a Riemannian manifold of dimension $n$. 
 If $\Psi\in\calC(M)$ has degree different from $n-1$ and is angular at $p$, then there exists $\psi\in \Omega^n(TM)$ with 
  $\Psi= [\psi]$  such that $\tau_p\psi$ has constant coefficients.
\end{lemma}

\begin{proof}As in the translation-invariant case one sees that there exists $\omega\in \Omega^n(TM)$ with $\Psi = [\omega]$. 
Since $\tau_p \Psi$ is angular, there exists  by Theorem~\ref{thm:ccc} a constant coefficient form $\theta$ on $TT_pM$ representing $\tau_p\Psi$. 
 Choose a geodesic coordinate neighborhood $U$ around $p$ to obtain for each $q\in U$ via parallel transport along  the unique geodesic connecting $p$ with $q$
 a linear isometry $P_q\colon T_qM \to T_pM$. Now $\eta_q = (P_q\oplus P_q)^*(\theta-\tau_p\omega)$ defines a smooth local section of $\Omega^n(TTM)^{tr}$; 
 modify $\eta$ by a smooth cut-off function to get a smooth global section. By construction, $\eta$ induces the zero curvature measure on each tangent space to $M$. 
 Put $\psi= \omega+ \tau^{-1}(\eta)$. Then $\tau_p\psi$ has constant coefficients and 
 $[\psi]=\tau^{-1}[\tau\omega+\eta] ={ \tau^{-1}[\tau\omega] } = \Psi$ by  Lemma~\ref{lemma:tau}. 
\end{proof}

\begin{proof}[Proof of Theorem~\ref{thm:pullback}]
By Lemma~\ref{lemma:char_angularity} it suffices to show that there exists a number 
$g(L)$   such that
for every closed convex cone $C$ in $T_pM$  with lineality space $L=L(C)$
 $$c_{\tau_p(f^*\Psi)}(C)=g(L) \gamma(C).$$  
Let $C\subset T_pM$  be some fixed closed convex cone.  
By continuity and finite additivity (Lemma~\ref{lemma:triangulation} and \eqref{eq:inclusion_exclusion}),
we may assume  without loss of generality  that $pr_{T_pM/L}(C)$ is a simplicial cone. We may also assume that $f$ is  the inclusion map $M\hookrightarrow \overline M$.  
By Lemma~\ref{lemma:constant_coeff_form} we may further assume that 
  $$\Psi(P,U)= \int_{N_1(P)\cap \pi^{-1}(U)} \omega$$
  with $\omega\in\Omega^n(T\overline M)$ where  $\tau_p\omega$ has constant coefficients.  
  
  Put $F= \exp^M(L\cap B(0,\varepsilon))$,  where $\exp^M\colon T_p M\to M$ denotes the exponential map, 
  and let $j\colon T F^\perp\to T\overline M$ denote the inclusion of the normal bundle of $F\subset \overline M$
  into the tangent bundle of $\overline M$. Denote by  $r_q(j^*\omega)$  the restriction of $j^*\omega$ to the fibers of $T F^\perp$.  
 By Lemma~\ref{lemma:bfs}, Equation \eqref{eq:pullback}, and Lemma~\ref{lemma:fiber_int_disc},  we obtain 
 \begin{align*}
  c_{\tau_p(\iota^* \Psi)} (C) 
  & = \lim_{r\to 0} \frac{1}{\omega_kr^k}    \Psi(\exp^M(C\cap  B(0,\varepsilon)), \exp^M(L \cap B(0,r)))\\
 & = \lim_{r\to 0} \frac{1}{\omega_kr^k}  
 \int_{  F\cap B(0,r)}   \int_{ N_1(\exp^M(C\cap  B(0,\varepsilon)))\cap \pi^{-1}(q)} r_q(j^*\omega)   \\
 & =  \vec T_pF\lrcorner  \int_{ N_1(\exp^M(C\cap  B(0,\varepsilon)))\cap \pi^{-1}(p)} r_p(j^*\omega)  \\
& = \vec T_pF\lrcorner  \int_{ C^\circ \cap D^n } r_p(j^*\omega) 
 \end{align*}  
 where $\vec{T}_pF = X_1\wedge \cdots \wedge X_k$ for some suitably oriented orthonormal basis of $T_p F$.

 Let $\nabla, \overline\nabla $ denote the Levi-Civita connections of $M\subset \overline M$. Using parallel transportation with respect to the normal 
 connection of $F\subset \overline M$, 
we extend the tangent vectors $X_1,\ldots, X_k$ to a  local orthonormal frame of $X_1,\ldots,X_n$ around $p$ in $\overline M$ with
$$X_1,\ldots, X_k\in T_q F\quad \text{and} \quad X_{k+1},\ldots, X_n\in T_q F^\perp$$
along $F$ and such that 
 \begin{equation}\label{eq:parallel}pr_{T_pF^\perp}(\overline\nabla_{X_i}X_j)=\nabla^\perp_{X_i}X_j= 0, \quad i=1,\ldots,k,\quad  j=k+1,\ldots, n\end{equation}
at $p$. Here $pr_{T_pF^\perp}$ denotes the orthogonal projection onto $T_pF^\perp$.

 Fix $\xi\in T F^\perp$ and extend $\xi$ to local section of $TF^\perp$ by declaring 
 $\xi = \sum_{i=k+1}^n \xi^i X_i$ with constant coefficients $\xi^{k+1},\ldots, \xi^n$
 Let $c_i\colon \RR\to F $ be a smooth curve with $c_i(0)=p, \dot c_i(0)=X_i$, $i=1,\ldots,k$. Then 
 $$\widetilde X_i := \dt \xi \circ c_i\in T_\xi TF^\perp, \qquad i=1,\ldots, k,$$
 is a lift of $X_i\in T_p F$ to the normal bundle. 
 Now
\begin{equation}\label{eq:decomp_lift}T_\xi TF^\perp\subset T_\xi T\overline M= H_\xi \oplus V_\xi\cong T_p \overline M \oplus T_p \overline M\end{equation}
 and thus each lift $\widetilde X_i$ decomposes as $X_i + X_i^V$ into horizontal and vertical components.   
 
 Let the frame $X_1,\ldots, X_n$ together with some  coordinate system $(x^1, \ldots , x^n)$  around $p\in \overline M$ define coordinates $x^1,\ldots,x^n, y^1, \ldots y^n$ on $T\overline M$.
 In terms of these  coordinates, the horizontal subspace $H_\xi$ is spanned by
 $$ \pder{}{x^i} + \sum_{j=1}^n\langle \xi, \overline{\nabla}_{\pder{}{x^i}} X^j\rangle \pder{}{y^j}, \qquad i=1,\ldots,n.$$
 Using \eqref{eq:parallel}, the vertical component of the lift $\widetilde  X_i$ is thus given by
 $$ X_i^V= -\sum_{j=1}^k \langle \xi , \overline \nabla_{X_i} X_j\rangle \frac{\partial}{\partial y^j}.$$
If $\xi\in T_pF^\perp \cap T_pM$, then by the Gauss formula \eqref{eq:gauss}
$$X_i^V = -\sum_{j=1}^k \langle \xi , \overline \nabla_{X_i} X_j\rangle \frac{\partial}{\partial y^j}
=-\sum_{j=1}^k \langle \xi , \nabla_{X_i} X_j\rangle \frac{\partial}{\partial y^j} = 0,$$
where the last equality holds since  the second fundamental form of $F\subset M$ vanishes at $p$ by the discussion surrounding \eqref{eq:acceleration}.
If $\xi\in T_pM^\perp $, then 
$$X_i^V = -\sum_{j=1}^k \langle \xi , \overline\nabla_{X_i} X_j\rangle \frac{\partial}{\partial y^j} = \sum_{j=1}^k \langle \xi , h(X_i, X_j)\rangle \frac{\partial}{\partial y^j}$$
where $h\in \Sym^2(T_p^*F)\otimes T_pF^\perp$ denotes the second 
fundamental form of $F\subset \overline M$ at $p$.
We conclude that 
 \begin{equation}\label{eq:vertical_comp} X_i^V=  \sum_{j=1}^k \langle pr_{T_pM^\perp}(\xi) , h(X_i, X_j)\rangle \frac{\partial}{\partial y^j}\end{equation}
 for all $ \xi\in T_pF^\perp$.

By  what was already said, the proof will be complete if we can show that there exists some constant $g(L)$ such that 
$$\vec T_pF\lrcorner  \int_{ C^\circ \cap D^n } r_p(j^*\omega) = g(L)\gamma(C)$$
for every closed convex cone $C\subset T_pM$ with lineality space equal to  $L$.
Now by the definition of $\tau_p\omega$ and \eqref{eq:decomp_lift} we have 
\begin{align*}\vec T_pF\lrcorner r_p(j^*\omega)(  \pder{}{y^{k+1}},\ldots,  \pder{y^{n}} )
  & = \omega_\xi( \pder{}{y^{k+1}},\ldots,  \pder{y^{n}} , \widetilde X_1, \ldots, \widetilde X_k) \\
& = \tau_p\omega( \pder{}{y^{k+1}},\ldots,  \pder{y^{n}}  , X_1+X_1^V, \ldots, X_k+X_k^V)
\end{align*}
 Since $\tau_p\omega$ has constant coefficients, the last expression is by \eqref{eq:vertical_comp} evidently a polynomial function $f(\xi)$ depending only on $pr_{T_pM^\perp}(\xi) $. Using that $C^\circ = C' \oplus T_pM^\perp$, 
where $C'$ is the polar cone of $C$ relative to $T_pM$, 
we compute 
\begin{align*}\vec T_pF\lrcorner  \int_{ C^\circ \cap D^n } r_p(j^*\omega) & = \int_{(C' \oplus T_pM^\perp) \cap D^n} f(\xi)\; d\xi\\ 
 &  = \int_{\{\eta\in C'\colon |\eta|\leq 1 \}} \int_{\{\zeta\in T_pM^\perp\colon |\zeta|^2\leq1 - |\eta|^2\}} f(\zeta)\; d\zeta\; d\eta\\
 & = \left(\int_0^1 \int_{\{\zeta\in T_pM^\perp\colon |\zeta|^2\leq1 -r^2\}} f(\zeta) \; d\zeta\; r^{m-k-1} dr\right)\\
 & \hspace{5cm}\times \vol_{k-1}(\{\eta\in C'\colon |\eta|= 1 \})\\
 & = g(L)\gamma(C).
\end{align*}
\end{proof}

\section{Proof of Theorem~\ref{conj:angularity}} \label{sec:proof conj}

With Theorem~\ref{thm:pullback} and Theorem~\ref{thm:ccc} at our disposal we are now ready to give a proof of the angularity conjecture. 
Our starting point is Lemma~\ref{lemma:angularity} below which is a slight generalization of \cite{bfs}*{Theorem~2.8}. First we need a simple 
lemma on cones.

\begin{lemma}\label{lemma:intersection_cone}
 Let $C\subset \RR^n$ be  a closed convex cone with lineality space $L$. If $v\in S^{n-1}$ is not orthogonal to $L$, then for every $s\in \RR$
\begin{equation}\label{eq:intersection_cone_hyperplane}C\cap (v^\perp + sv)= (C\cap v^\perp) + s u,\end{equation}
where $u = pr_{L}(v)/|pr_L(v)|^2$.  
Moreover, the lineality space of $C\cap v^\perp$ is $L\cap v^\perp$ and if  $pr_{\RR^n/L}(C)$ is  simplicial then also
 $pr_{\RR^n/(L\cap v^\perp)}(C\cap v^\perp)$  is simplicial. 
\end{lemma}
\begin{proof} It is straightforward to verify the inclusions $ C\cap (v^\perp + sv)\supset  (C\cap v^\perp) + s u$ and $C\cap (v^\perp + sv)-su \subset  (C\cap v^\perp)$; this implies 
\eqref{eq:intersection_cone_hyperplane}. It is obvious that the lineality space of $C\cap v^\perp$ is $L\cap v^\perp$.

Next we claim that $pr_{\RR^n/L}(C)= 
pr_{\RR^n/L}(C\cap v^\perp)$. To see this, given $x\in C$ put $s=\langle x,v\rangle$. But then $x\in (C\cap v^\perp)+ su$ by $\eqref{eq:intersection_cone_hyperplane}$ and so $pr_{\RR^n/L}(C)\subset 
pr_{\RR^n/L}(C\cap v^\perp)$. Since the reverse inclusion trivially holds, we have reached the desired conclusion. As a consequence, if  $pr_{\RR^n/L}(C)$ is  simplicial,
then there exist elements  $v_1,\ldots, v_m\in C\cap v^\perp$  that a linearly independent modulo $L$ such that every $x\in C\cap v^\perp$ can be expressed as 
$x = \sum_{i=1}^m \lambda_i v_i + l$
with $l\in L$ and  nonnegative numbers $\lambda_1,\ldots,\lambda_m$. It follows
that $l\in L\cap v^\perp$ and hence $pr_{\RR^n/(L\cap v^\perp)}(C\cap v^\perp)$  is simplicial. 
\end{proof}

\begin{lemma} \label{lemma:angularity}
  Let $\Psi\in \calC(\RR^n)$ be a smooth curvature measure. If $\Psi$ is angular at $p$, then $V_1\cdot \Psi$ is angular at $p$ as well.
\end{lemma}
\begin{proof}
By Lemma~\ref{lemma:char_angularity} the assumption implies that there exists a constant $f(L)$ such that 
$$c_{\tau_p\Psi}(C)=f(L)\gamma(C)$$
for every closed convex cone $C\subset \RR^n$ with $L=L(C)$.
We have to show that there exists a constant 
$h(L)$ such that   
 $$c_{\tau_p(V_1\cdot\Psi)}(C)= h(L)\gamma(C)$$
for every closed convex cone $C\subset \RR^n$ with  $L=L(C)$. By continuity and finite additivity (Lemma~\ref{lemma:triangulation} and \eqref{eq:inclusion_exclusion}),
we may assume  without loss of generality  that $pr_{\RR^n/L}(C)$ is a simplicial cone and that $p=0$. 
 
By Lemma~\ref{lemma:bfs} and \eqref{eq:muk action}
 \begin{align}
 c_{\tau_p(V_1\cdot\Psi)}(C) \notag
  & =\lim_{r\to 0} \frac{1}{\omega_{k} r^{k}} \int_{\overline{\Grass}_{n-1}} \Psi(C \cap H, B(0,r)\cap L\cap H) \; dH\\
\label{eq:limit}  & = \frac{1}{2\omega_{n-1}} \int_{S^{n-1}}  \lim_{r\to 0}\frac{1}{\omega_k r^k}\int_\RR   \Psi(C \cap H_s, B(0,r)\cap L\cap H_s)\; ds\;  dv
\end{align}
where $H_s =sv+ v^\perp$. Assume now that $v$ is not orthogonal to $L$. 
By Lemma~\ref{lemma:intersection_cone} the cone $C\cap H_s$ is a submanifold with corners and so we may apply
Lemma~\ref{lemma:fiber_int} and \eqref{eq:intersection_cone_hyperplane} to obtain that  the inner integral in \eqref{eq:limit} equals 
$$ \int_\RR \int_{L\cap B(0,r)\cap H_s}  \left( L\cap v^\perp\lrcorner \int_{N(C\cap v^\perp)\cap \pi^{-1}(y-su)} r_{y-su}(j^*\beta_{su}^*\omega)\right)\; dy\;ds,$$ 
where $r_x$ denotes the restriction to the fibers of the unit normal bundle of $L\cap v^\perp\subset \RR^n$, $u= pr_{L}(v)/|pr_L(v)|^2$,
and  $\beta_{su}\colon \RR^n\to\RR^n$ denotes translation by $su$. 
Since
$$ \int_\RR \int_{L\cap B(0,r)\cap H_s}  = \int_{L\cap B(0,r)} \cos\theta,$$
with $\theta$ denoting the angle between $v$ and $L$, passing to the limit yields
\begin{align*} c_{\tau_p(V_1\cdot\Psi)}(C) & = \frac{1}{2\omega_{n-1}} \int_{S^{n-1}}\cos\theta\left( L\cap v^\perp\lrcorner \int_{N(C\cap v^\perp)\cap \pi^{-1}(0)}r_{0}(j^*\omega) \right) \;  dv\\ 
& = \frac{1}{2\omega_{n-1}}\int_{S^{n-1}}  \cos\theta \;c_{\tau_p\Psi}(C\cap v^\perp) \;  dv\\
  & =  \frac{1}{2\omega_{n-1}}\int_{S^{n-1}}  \cos\theta \;f(L\cap v^\perp) \;\gamma( C\cap v^\perp) \;  dv
 \end{align*}
Now the same computation   as in the proof of \cite{bfs}*{Theorem 2.8} shows  that there exists a constant $h(L)$ so that the latter integral equals 
$h(L)\gamma(C)$ for every closed convex cone $C$ with lineality space $L$. This completes the proof.
\end{proof}

\begin{lemma} \label{lemma:constant_coeff_form_global} Let $M$ be a Riemannian manifold of dimension $n$. 
 If $\Psi$ is an angular curvature measure on $M$ of degree different from $n-1$, then there exists $\psi\in \Omega^n(TM)$ with 
  $\Psi= [\psi]$  such that $\tau_p\psi$ has constant coefficients for all $p\in M$.
\end{lemma}

\begin{proof} Choose a local orthonormal frame for $M$, to obtain a local trivialization $\varphi\colon \pi^{-1}(U)\to U\times \RR^n$ of the tangent 
 bundle of $M$ which is a linear isometry when restricted to the fibers. Consider also the corresponding trivialization $\widetilde \varphi\colon \Curv(TM)|_U\to 
 U\times \Curv(\RR^n)$. By hypothesis, $\tau\Psi$ is a smooth section of $\Curv(TM)$ such that each $\tau_p\Psi$ is angular. 
 In terms of our local trivialization, this gives  angular curvature measures $\widetilde\varphi(\tau_p\Psi)$ on $\RR^n$ which  depend smoothly on $p$. 
 By Theorem~\ref{thm:ccc}, we may choose elements $\theta_p\in \largewedge^n(\RR^n\oplus \RR^n)^*$ which depend smoothly on $p$ and 
 satisfy $[\theta_p]=\tau_p \Psi$. Put $\psi=\tau^{-1}\theta$. By construction, each $\tau_p\psi$ has constant coefficients and 
 Lemma~\ref{lemma:tau} guarantees that $\Psi=[\psi]$ on $U$. Now choose a partition of unity to get a globally defined form $\psi$ with the desired properties. 
\end{proof}

Recall that $\pi_k$ denotes the projection to the degree $k$ component of $\calC(M)$.

\begin{lemma}\label{lemma:transfer restriction} Let  $\iota\colon M \to\overline M$ be a (not necessarily isometric)
immersion of Riemannian manifolds and let  $\Psi\in \calC_k(\overline M)$. Then 
 $$\pi_k \circ \tau \circ \iota^*(\Psi)= (d\iota)^*\circ \tau \circ \pi_k(\Psi)$$
\end{lemma}
\begin{proof} Recall from \cite{sw_spheres} the canonical map  $\Lambda_k'\colon \calC_k(M)\to \Gamma(\Curv_k(TM))$ defined as follows.
Let $p\in M$, and fix a local diffeomorphism $\phi:T_pM\to M$ with $\phi(0)=p, d\phi_0=\mathrm{id}$. 
For $t\in\RR$ let $h_t(y)=ty, y\in T_pM$. For $\Psi\in\calC_k(M)$ set
$$\Lambda_k'(\Psi)|_x:=\lim_{t\to 0}\frac{1}{t^k} (\phi\circ h_t)^*\Psi.$$
By \cite{sw_spheres}*{Proposition~3.5} we have 
 $$\pi_k\circ \tau(\Psi)=\Lambda'_k(\Psi).$$ 
 Thus we will reach the desired conclusion if we can prove 
 \begin{equation}\label{eq:pullback without metric}\Lambda'_k \circ \iota^* (\Psi) = (d\iota^*) \circ \Lambda'_k(\Psi).\end{equation}
 for every immersion $\iota\colon M\to\overline M$. 
 Now in the case where $\iota\colon \RR^m\to \RR^n$ is the standard inclusion the validity of \eqref{eq:pullback without metric} 
 is clear. The general case can be reduced to this situation if
 \eqref{eq:pullback without metric} is already established for diffeomorphisms.  This case is straightforward to verify. 
\end{proof}

\begin{lemma}\label{lemma:extension}
 Let $\iota\colon M\to \RR^n$ be an isometric embedding of an $m$-dimensional Riemannian manifold $M$ and let $q$ be some fixed point of $M$.
 If $\Phi$ is an angular curvature measure on $M$,
 then there exists $\Psi\in \calC(\RR^n)$ which is angular at all points of $M$ such that 
 $\iota^*\Psi$ coincides with $\Phi$ up to an element of  $\calC_{m-1}(M)$ in an open neighborhood of $q$.
\end{lemma}
\begin{proof} It clearly suffices to prove the lemma for curvature measures of pure degree.  If $\Phi$ has degree $m-1$ or $m$, then there is nothing to prove. Assume therefore  that  $\Phi$ has degree $k<m-1$
and that the lemma has already been established for  curvature measure 
of degree greater than $k$. 
 Since $\Phi$ is angular and has degree $k<m-1$, there exists by Lemma~\ref{lemma:constant_coeff_form_global} an $m$-form $\phi\in \Omega^m(TM)$ 
 such that $\Phi=[\phi]$ and $\tau_p\phi\in \Omega^m(T_pM\oplus T_pM)^{tr}$ has constant coefficients for all $p\in M$. 
 
Choose  $\theta_p\in \largewedge^{n-m}((T_pM)^\perp)^*$ depending smoothly on $p$ and with $|\theta_p|=1$ in some neighborhood of $q$. 
Now put 
$$\omega_p= (pr_{T_pM}\times pr_{T_pM})^* \tau_p \phi \wedge (\pi_2\circ pr_{T_pM^\perp})^* \theta_p\in \Omega^n(T_p\RR^n\oplus T_p\RR^n)^{tr}.$$
Here $pr_{T_pM}\colon T_p\RR^n\to T_p M$ denotes the orthogonal projection and $\pi_2\colon T_p\RR^n\oplus T_p\RR^n\to T_p\RR^n$ the  projection 
to the second factor. By construction, $\omega_p$ has constant coefficients. Now choose some smooth form $\widetilde \omega\in \Omega^n(T\RR^n)$ 
of bi-degree $(k,n-k)$ such that
$\tau_p\widetilde \omega$ has constant coefficients at points of $M$ and  coincides with $\omega_p$ in some neighborhood of $U\subset M$ of $q$. 
Let $\Psi$ be the curvature measure on $\RR^n$ defined by $\widetilde\omega$.

The normal disc current of $P\subset T_pM$ inside $T_p\RR^n$ can be constructed from the normal disc current of $P$ inside $T_pM$ via
$$N_1^{T_p\RR^n} (P)= \{(x,u+v)\colon (x,u)\in N_1^{T_pM}(P), \ v\in T_pM^\perp, \ |v|^2\leq 1-|u|^2\}.$$
Hence, if $f(x,u,v)= (x,u+\sqrt{1-|u|^2}v)$, then
$$ N_1^{T_p\RR^n} (P)= f_*(N_1^{T_pM}(P)\times D^{n-m})$$ and so, by Lemma~\ref{lemma:tau}, for $P\subset T_pM$ 
\begin{align*}
  \tau_p\Psi(P,V) & = \int_{N_1^{T_p\RR^n}(P)\cap \pi^{-1}(V)} \tau_p\widetilde\omega\\
 & = \int_{N_1^{T_pM}(P)\cap \pi^{-1}(V)} \int_{D^{n-m}} f^* \omega\\
 & = \omega_{n-m}  \int_{N_1^{T_pM}(P)\cap \pi^{-1}(V)} (1-|u|^2)^{(n-m)/2} \tau_p\phi.
\end{align*}
Since $\tau_p\phi$ has constant coefficients, the last integral is a non-zero  constant multiple of  $ \Phi(P,V)$. 
Adjusting $\Psi$ by some non-zero constant, we get 
 \begin{equation}\label{eq:pullback_angular} (d\iota_p)^*\tau_p\Psi= \tau_p \Phi, \qquad \text{for all }p\in U.\end{equation}
 By Theorem~\ref{thm:pullback},  $\iota^*\Psi-\Phi$ is angular and, since the pullback preserves the filtration on curvature measures, lies 
 in $\calC_k(M)$. But by Lemma~\ref{lemma:transfer restriction} and \eqref{eq:pullback_angular}, $\iota^*\Psi-\Phi$  has no non-trivial component of degree less than $k+1$.
 By induction, this finishes the proof.
 \end{proof}

\begin{proof}[Proof of Theorem~\ref{conj:angularity}] 

Choose an isometric embedding of  $M$ into some Euclidean space $ \RR^n$. Fix some point $q$ in $M$.
By Lemma~\ref{lemma:extension} there exist
$\Psi \in\calC(\RR^n)$  and $\Xi\in \calC_{m-1}(M)$
such  that $\Psi$ is angular at  points of $ M$ and $\Phi$ coincides with $\iota^*\Psi + \Xi$ 
in some neighborhood $W\subset M$ of $q$.  By Theorem~\ref{thm:product} the Alesker product commutes with pullback and hence
 \begin{align*}
  (V_1 \cdot \Phi)|_W & = V_1|_W \cdot \Phi|_W\\
  & = V_1|_W \cdot (\iota^* \Psi + \Xi)|_W\\
  & = (V_1\cdot \iota^*\Psi)|_W + (V_1\cdot \Xi)|_W\\
   & = (\iota^*(V_1\cdot\Psi) + V_1\cdot \Xi)|_W,
 \end{align*}
 where the last equality follows from the invariance of the Lipschitz-Killing valuations 
 under pullback by isometric immersions. 
 Since the Alesker product preserves the filtration on curvature measures and every curvature measure in $\calC_{m-1}(M)$ is angular, we conclude that $V_1 \cdot \Xi$ is angular. 
 Lemma~\ref{lemma:angularity}
 implies that  $V_1\cdot \Psi$ is angular at points of  $M$. Together with Theorem~\ref{thm:pullback} this implies that 
 $\iota^*(V_1\cdot \Psi)$ is angular. We conclude that 
 $V_1\cdot \Phi$ is angular in $W$.   Since $q$ was arbitrary and $ V_k= \frac{2^k}{k!\omega_k} V_1^k$, this completes the proof. 
 \end{proof}


\begin{bibdiv}
\begin{biblist}

 \setlength\itemsep{.15cm}
 




\bib{alesker_irred}{article}{
   author={Alesker, Semyon},
   title={Description of translation invariant valuations on convex sets
   with solution of P. McMullen's conjecture},
   journal={Geom. Funct. Anal.},
   volume={11},
   date={2001},
   number={2},
   pages={244--272},
   issn={1016-443X},
   review={\MR{1837364}},
   doi={10.1007/PL00001675},
}



\bib{alesker_hard_lefschetz}{article}{
   author={Alesker, Semyon},
   title={Hard Lefschetz theorem for valuations, complex integral geometry,
   and unitarily invariant valuations},
   journal={J. Differential Geom.},
   volume={63},
   date={2003},
   number={1},
   pages={63--95},
   issn={0022-040X},
   review={\MR{2015260}},
}
\bib{alesker_product}{article}{
   author={Alesker, Semyon},
   title={The multiplicative structure on continuous polynomial valuations},
   journal={Geom. Funct. Anal.},
   volume={14},
   date={2004},
   number={1},
   pages={1--26},
   issn={1016-443X},
   review={\MR{2053598}},
   doi={10.1007/s00039-004-0450-2},
}



\bib{valmfdsI}{article}{
   author={Alesker, Semyon},
   title={Theory of valuations on manifolds. I. Linear spaces},
   journal={Israel J. Math.},
   volume={156},
   date={2006},
   pages={311--339},
   issn={0021-2172},
   review={\MR{2282381}},
   doi={10.1007/BF02773837},
}

\bib{valmfdsII}{article}{
   author={Alesker, Semyon},
   title={Theory of valuations on manifolds. II},
   journal={Adv. Math.},
   volume={207},
   date={2006},
   number={1},
   pages={420--454},
   issn={0001-8708},
   review={\MR{2264077}},
   doi={10.1016/j.aim.2005.11.015},
}

\bib{valmfdsIV}{article}{
   author={Alesker, Semyon},
   title={Theory of valuations on manifolds. IV. New properties of the
   multiplicative structure},
   conference={
      title={Geometric aspects of functional analysis},
   },
   book={
      series={Lecture Notes in Math.},
      volume={1910},
      publisher={Springer, Berlin},
   },
   date={2007},
   pages={1--44},
}




\bib{valmfds_survey}{article}{
   author={Alesker, Semyon},
   title={Theory of valuations on manifolds: a survey},
   journal={Geom. Funct. Anal.},
   volume={17},
   date={2007},
   number={4},
   pages={1321--1341},
   issn={1016-443X},
   review={\MR{2373020}},
   doi={10.1007/s00039-007-0631-x},
}
		
 




		



\bib{valmfdsIG}{article}{
   author={Alesker, Semyon},
   title={Valuations on manifolds and integral geometry},
   journal={Geom. Funct. Anal.},
   volume={20},
   date={2010},
   number={5},
   pages={1073--1143},
   issn={1016-443X},
   review={\MR{2746948}},
   doi={10.1007/s00039-010-0088-1},
}



\bib{alesker_fourier}{article}{
   author={Alesker, Semyon},
   title={A Fourier-type transform on translation-invariant valuations on
   convex sets},
   journal={Israel J. Math.},
   volume={181},
   date={2011},
   pages={189--294},
   issn={0021-2172},
   review={\MR{2773042}},
   doi={10.1007/s11856-011-0008-6},
}




\bib{ab_product}{article}{
   author={Alesker, Semyon},
   author={Bernig, Andreas},
   title={The product on smooth and generalized valuations},
   journal={Amer. J. Math.},
   volume={134},
   date={2012},
   number={2},
   pages={507--560},
   issn={0002-9327},
   review={\MR{2905004}},
   doi={10.1353/ajm.2012.0011},
}



\bib{valmfdsIII}{article}{
   author={Alesker, Semyon},
   author={Fu, Joseph H. G.},
   title={Theory of valuations on manifolds. III. Multiplicative structure
   in the general case},
   journal={Trans. Amer. Math. Soc.},
   volume={360},
   date={2008},
   number={4},
   pages={1951--1981},
   issn={0002-9947},
   review={\MR{2366970}},
   doi={10.1090/S0002-9947-07-04489-3},
}

\bib{crm_bcn}{book}{
   author={Alesker, Semyon},
   author={Fu, Joseph H. G.},
   title={Integral geometry and valuations},
   series={Advanced Courses in Mathematics. CRM Barcelona},
   publisher={Birkh\"auser/Springer, Basel},
   date={2014},
   pages={viii+112},
   isbn={978-3-0348-0873-6},
   isbn={978-3-0348-0874-3},
}
		


		
\bib{bf_indefinite}{article}{
   author={Bernig, Andreas},
   author={Faifman, Dmitry},
   title={Valuation theory of indefinite orthogonal groups},
   journal={J. Funct. Anal.},
   volume={273},
   date={2017},
   number={6},
   pages={2167--2247},
   issn={0022-1236},
   review={\MR{3669033}},
   doi={10.1016/j.jfa.2017.06.005},
}
	
\bib{bf_convolution}{article}{
   author={Bernig, Andreas},
   author={Fu, Joseph H. G.},
   title={Convolution of convex valuations},
   journal={Geom. Dedicata},
   volume={123},
   date={2006},
   pages={153--169},
   issn={0046-5755},

}





\bib{hig}{article}{
   author={Bernig, Andreas},
   author={Fu, Joseph H. G.},
   title={Hermitian integral geometry},
   journal={Ann. of Math. (2)},
   volume={173},
   date={2011},
   pages={907--945},
}


\bib{bfs}{article}{
   author={Bernig, A.},
   author={Fu, J. H. G.},
   author={Solanes, G.},
   title={Integral geometry of complex space forms},
   journal={Geom. Funct. Anal.},
   volume={24},
   date={2014},
   number={2},
   pages={403--492},
   issn={1016-443X},
   review={\MR{3192033}},
   doi={10.1007/s00039-014-0251-1},
}





\bib{bfs_dual}{article}{
   author={Bernig, Andreas},
   author={Fu, Joseph H. G.},
   author={Solanes, Gil},
   title={Dual curvature measures in Hermitian integral geometry},
   conference={
      title={Analytic aspects of convexity},
   },
   book={
      series={Springer INdAM Ser.},
      volume={25},
      publisher={Springer, Cham},
   },
   date={2018},
   pages={1--17},
   review={\MR{3753099}},
}


\bib{blaschke39}{article}{
   author={Blaschke, W.},
   title={Densita negli spazi di Hermite},
   journal={Rendiconti dell' Academia dei Lincei},
   volume={29},
   date={1939},
   pages={105--108},
}
		
\bib{cms}{article}{
   author={Cheeger, Jeff},
   author={M\"{u}ller, Werner},
   author={Schrader, Robert},
   title={On the curvature of piecewise flat spaces},
   journal={Comm. Math. Phys.},
   volume={92},
   date={1984},
   number={3},
   pages={405--454},
   issn={0010-3616},
   review={\MR{734226}},
}
	
\bib{donnelly}{article}{
   author={Donnelly, Harold},
   title={Heat equation and the volume of tubes},
   journal={Invent. Math.},
   volume={29},
   date={1975},
   number={3},
   pages={239--243},
   issn={0020-9910},
   review={\MR{0402832}},
   doi={10.1007/BF01389852},
}






\bib{faifman_crofton}{article}{
   author={Faifman, Dmitry},
   title={Crofton formulas and indefinite signature},
   journal={Geom. Funct. Anal.},
   volume={27},
   date={2017},
   number={3},
   pages={489--540},
   issn={1016-443X},
   review={\MR{3655955}},
   doi={10.1007/s00039-017-0406-y},
}

\bib{faifman_contact}{article}{
   author={Faifman, Dmitry},
   title={Contact integral geometry and the Heisenberg algebra},
   eprint={arXiv:1712.09313},
}


\bib{federer}{article}{
   author={Federer, H.},
   title={Curvature measures},
   journal={Trans. Amer. Math. Soc.},
   volume={93},
   date={1959},
   pages={418--491},
   issn={0002-9947},
   review={\MR{0110078}},
   doi={10.2307/1993504},
}



\bib{fu_subanalytic}{article}{
   author={Fu, Joseph H. G.},
   title={Curvature measures of subanalytic sets},
   journal={Amer. J. Math.},
   volume={116},
   date={1994},
   number={4},
   pages={819--880},
   issn={0002-9327},
   review={\MR{1287941}},
   doi={10.2307/2375003},
}


		

\bib{fu_intersection}{article}{
   author={Fu, Joseph H. G.},
   title={Intersection theory and the Alesker product},
   journal={Indiana Univ. Math. J.},
   volume={65},
   date={2016},
   number={4},
   pages={1347--1371},
   issn={0022-2518},
   review={\MR{3549204}},
   doi={10.1512/iumj.2016.65.5846},
}


\bib{fu_sandbjerg}{article}{
   author={Fu, Joseph H. G.},
   title={Integral geometric regularity},
   conference={
      title={Tensor valuations and their applications in stochastic geometry
      and imaging},
   },
   book={
      series={Lecture Notes in Math.},
      volume={2177},
      publisher={Springer, Cham},
   },
   date={2017},
   pages={261--299},
}


\bib{fpr_wdc}{article}{
   author={Fu, J. H. G.},
   author={Pokorn\'y, Du\v san},
   author={Rataj, Jan},
   title={Kinematic formulas for sets defined by differences of convex
   functions},
   journal={Adv. Math.},
   volume={311},
   date={2017},
   pages={796--832},
   issn={0001-8708},
   review={\MR{3628231}},
   doi={10.1016/j.aim.2017.03.003},
}


\bib{fw_Riemannian}{article}{
   author={Fu, Joseph H. G.},
   author={Wannerer, Thomas},
   title={Riemannian curvature measures},
   journal={Geom. Funct. Anal.},
   status={in press},
   eprint={arxiv:1711.02155},
}
 



\bib{fulton_harris}{book}{
   author={Fulton, W.},
   author={Harris, J.},
   title={Representation theory},
   series={Graduate Texts in Mathematics},
   volume={129},
   note={A first course;
   Readings in Mathematics},
   publisher={Springer-Verlag, New York},
   date={1991},
   pages={xvi+551},
   isbn={0-387-97527-6},
   isbn={0-387-97495-4},
   review={\MR{1153249}},
   doi={10.1007/978-1-4612-0979-9},
}



	


%
   

   


\bib{gray_tubes}{book}{
   author={Gray, Alfred},
   title={Tubes},
   series={Progress in Mathematics},
   volume={221},
   edition={2},
   note={With a preface by Vicente Miquel},
   publisher={Birkh\"{a}user Verlag, Basel},
   date={2004},
   pages={xiv+280},
   isbn={3-7643-6907-8},
   review={\MR{2024928}},
   doi={10.1007/978-3-0348-7966-8},
}
		   

\bib{howe_etal}{article}{
   author={Howe, Roger},
   author={Tan, Eng-Chye},
   author={Willenbring, Jeb F.},
   title={Stable branching rules for classical symmetric pairs},
   journal={Trans. Amer. Math. Soc.},
   volume={357},
   date={2005},
   number={4},
   pages={1601--1626},
   issn={0002-9947},
   review={\MR{2115378}},
   doi={10.1090/S0002-9947-04-03722-5},
}

\bib{klain_short}{article}{
   author={Klain, Daniel A.},
   title={A short proof of Hadwiger's characterization theorem},
   journal={Mathematika},
   volume={42},
   date={1995},
   number={2},
   pages={329--339},
   issn={0025-5793},
   review={\MR{1376731}},
   doi={10.1112/S0025579300014625},
}





\bib{klain_rota}{book}{
   author={Klain, Daniel A.},
   author={Rota, Gian-Carlo},
   title={Introduction to geometric probability},
   series={Lezioni Lincee. [Lincei Lectures]},
   publisher={Cambridge University Press, Cambridge},
   date={1997},
   pages={xiv+178},
   isbn={0-521-59362-X},
   isbn={0-521-59654-8},
   review={\MR{1608265}},
}



	









\bib{ludwig_minkowski}{article}{
   author={Ludwig, Monika},
   title={Minkowski valuations},
   journal={Trans. Amer. Math. Soc.},
   volume={357},
   date={2005},
   number={10},
   pages={4191--4213},
   issn={0002-9947},
   review={\MR{2159706}},
   doi={10.1090/S0002-9947-04-03666-9},
}

\bib{ludwig_intersection}{article}{
   author={Ludwig, Monika},
   title={Intersection bodies and valuations},
   journal={Amer. J. Math.},
   volume={128},
   date={2006},
   number={6},
   pages={1409--1428},
   issn={0002-9327},
   review={\MR{2275906}},
}



\bib{ludwig_HT}{article}{
   author={Ludwig, Monika},
   title={Minkowski areas and valuations},
   journal={J. Differential Geom.},
   volume={86},
   date={2010},
   number={1},
   pages={133--161},
   issn={0022-040X},
   review={\MR{2772547}},
}
		

\bib{lr_affine}{article}{
   author={Ludwig, Monika},
   author={Reitzner, Matthias},
   title={A characterization of affine surface area},
   journal={Adv. Math.},
   volume={147},
   date={1999},
   number={1},
   pages={138--172},
   issn={0001-8708},
   review={\MR{1725817}},
   doi={10.1006/aima.1999.1832},
}



\bib{lyz_proj}{article}{
   author={Lutwak, Erwin},
   author={Yang, Deane},
   author={Zhang, Gaoyong},
   title={$L_p$ affine isoperimetric inequalities},
   journal={J. Differential Geom.},
   volume={56},
   date={2000},
   number={1},
   pages={111--132},
   issn={0022-040X},
   review={\MR{1863023}},
}
		


\bib{lz_intersect}{article}{
   author={Lutwak, Erwin},
   author={Zhang, Gaoyong},
   title={Blaschke-Santal\'{o} inequalities},
   journal={J. Differential Geom.},
   volume={47},
   date={1997},
   number={1},
   pages={1--16},
   issn={0022-040X},
   review={\MR{1601426}},
}

\bib{mcmullen_angle-sum}{article}{
   author={McMullen, P.},
   title={Non-linear angle-sum relations for polyhedral cones and polytopes},
   journal={Math. Proc. Cambridge Philos. Soc.},
   volume={78},
   date={1975},
   number={2},
   pages={247--261},
   issn={0305-0041},
   review={\MR{0394436}},
   doi={10.1017/S0305004100051665},
}

\bib{mcmullen_valuations}{article}{
   author={McMullen, P.},
   title={Valuations and Euler-type relations on certain classes of convex
   polytopes},
   journal={Proc. London Math. Soc. (3)},
   volume={35},
   date={1977},
   number={1},
   pages={113--135},
   issn={0024-6115},
   review={\MR{0448239}},
   doi={10.1112/plms/s3-35.1.113},
}
	



\bib{milnor}{book}{
   author={Milnor, John},
   title={Collected papers. Vol. 1},
   note={Geometry},
   publisher={Publish or Perish, Inc., Houston, TX},
   date={1994},
   pages={x+295},
   isbn={0-914098-30-6},
   review={\MR{1277810}},
}


\bib{oneill}{book}{
   author={O'Neill, Barrett},
   title={Semi-Riemannian geometry},
   series={Pure and Applied Mathematics},
   volume={103},
   note={With applications to relativity},
   publisher={Academic Press, Inc. [Harcourt Brace Jovanovich, Publishers],
   New York},
   date={1983},
   pages={xiii+468},
   isbn={0-12-526740-1},
   review={\MR{719023}},
}



\bib{pr_wdc}{article}{
   author={Pokorn\'y, Du\v san},
   author={Rataj, Jan},
   title={Normal cycles and curvature measures of sets with d.c. boundary},
   journal={Adv. Math.},
   volume={248},
   date={2013},
   pages={963--985},
   issn={0001-8708},
   review={\MR{3107534}},
   doi={10.1016/j.aim.2013.08.022},
}



		

		
		
		

		
		

\bib{rohde40}{article}{
   author={Rohde, H.},
   title={Integralgeometrie 33. Unit\"are Integralgeometrie},
   language={German},
   journal={Abh. Math. Sem. Hansischen Univ.},
   volume={13},
   date={1940},
   pages={295--318},
   review={\MR{0002199 (2,12h)}},
}

\bib{santalo52}{article}{
   author={Santal{\'o}, L. A.},
   title={Integral geometry in Hermitian spaces},
   journal={Amer. J. Math.},
   volume={74},
   date={1952},
   pages={423--434},
   issn={0002-9327},
   review={\MR{0048062 (13,971g)}},
}
		
\bib{schneider_kinematische}{article}{
   author={Schneider, Rolf},
   title={Kinematische Ber\"uhrma\ss e f\"ur konvexe K\"orper},
   language={German},
   journal={Abh. Math. Sem. Univ. Hamburg},
   volume={44},
   date={1975},
   pages={12--23 (1976)},
   issn={0025-5858},
   review={\MR{0394447}},
   doi={10.1007/BF02992942},
}
\bib{schneider_curvature}{article}{
   author={Schneider, Rolf},
   title={Curvature measures of convex bodies},
   journal={Ann. Mat. Pura Appl. (4)},
   volume={116},
   date={1978},
   pages={101--134},
   issn={0003-4622},
   review={\MR{506976}},
   doi={10.1007/BF02413869},
}
		


\bib{schneiderBM}{book}{
   author={Schneider, Rolf},
   title={Convex bodies: the Brunn-Minkowski theory},
   series={Encyclopedia of Mathematics and its Applications},
   volume={151},
   edition={Second expanded edition},
   publisher={Cambridge University Press, Cambridge},
   date={2014},
   pages={xxii+736},
   isbn={978-1-107-60101-7},
   review={\MR{3155183}},
}






\bib{sw_even}{article}{
   author={Schuster, Franz E.},
   author={Wannerer, Thomas},
   title={Even Minkowski valuations},
   journal={Amer. J. Math.},
   volume={137},
   date={2015},
   number={6},
   pages={1651--1683},
   issn={0002-9327},
   review={\MR{3432270}},
   doi={10.1353/ajm.2015.0041},
}

\bib{sw_gen}{article}{
   author={Schuster, Franz E.},
   author={Wannerer, Thomas},
   title={Minkowski valuations and generalized valuations},
   journal={J. Eur. Math. Soc. (JEMS)},
   volume={20},
   date={2018},
   number={8},
   pages={1851--1884},
}



\bib{sw_sig}{book}{
   author={Schneider, Rolf},
   author={Weil, Wolfgang},
   title={Stochastic and integral geometry},
   series={Probability and its Applications (New York)},
   publisher={Springer-Verlag, Berlin},
   date={2008},
   pages={xii+693},
   isbn={978-3-540-78858-4},
   review={\MR{2455326}},
   doi={10.1007/978-3-540-78859-1},
}

\bib{shifrin81}{article}{
   author={Shifrin, T.},
   title={The kinematic formula in complex integral geometry},
   journal={Trans. Amer. Math. Soc.},
   volume={264},
   date={1981},
   number={2},
   pages={255--293},
   issn={0002-9947},
   review={\MR{603763 (83a:53067a)}},
   doi={10.2307/1998539},
}


\bib{sw_spheres}{article}{
   author={Solanes, Gil},
   author={Wannerer, Thomas},
   title={Integral geometry of exceptional spheres},
   journal={J. Differential Geom.},
   status={in press},
   eprint={arXiv:1708.05861},
}
 

	

\bib{strichartz}{article}{
   author={Strichartz, Robert S.},
   title={The explicit Fourier decomposition of $L^{2}({\rm SO}(n)/{\rm
   SO}(n-m))$},
   journal={Canad. J. Math.},
   volume={27},
   date={1975},
   pages={294--310},
   issn={0008-414X},
   review={\MR{0380277}},
   doi={10.4153/CJM-1975-036-x},
}



\bib{varga39}{article}{
   author={Varga, O.},
   title={\"Uber die Integralinvarianten, die zu einer Kurve in der Hermiteschen Geometrie geh\"oren},
   language={German},
   journal={Acta Litt. Scientarum Szeged},
   volume={9},
   date={1939},
   pages={88--102},
} 




\bib{weyl_tubes}{article}{
   author={Weyl,  H. },
   title={On the volume of tubes},
   journal= {Duke Math. J.},
   volume={61},
   date= {1939},
   pages={461--472}
}

\bib{zahle}{article}{
   author={Z\"ahle,  M. },
   title={Curvature and currents for unions of sets with positive reach},
   journal= {Geom. Ded.},
   volume={23},
   date= {1987},
   pages={155--172},
}
 


\end{biblist}
\end{bibdiv}
\end{document}